\documentclass[11pt]{amsart}
\usepackage{epsfig,graphicx, subfigure,epstopdf}
\usepackage{amsmath, amssymb,mathabx,mathrsfs}
\usepackage{color}
\usepackage[toc,page]{appendix}
\usepackage{comment}

\newtheorem{thm}{Theorem}[section]
\newtheorem{cor}[thm]{Corollary}
\newtheorem{lem}[thm]{Lemma}
\newtheorem{prop}[thm]{Proposition}
\theoremstyle{definition}
\newtheorem{defn}[thm]{Definition}
\theoremstyle{remark}
\newtheorem{rem}[thm]{Remark}
\newtheorem{ex}[thm]{Example}
\numberwithin{equation}{section}

\theoremstyle{definition}

\newcommand{\bd}{\textrm{bd}}
\newcommand{\inte}{\textrm{int}}
\newcommand{\ind}{\textrm{ind}}

\def\R{{\mathbb R}}

\begin{document}


\title[Combinatorial approach to detection of  fixed points]{Combinatorial approach to detection of fixed points, periodic orbits, and symbolic dynamics}

\author{Marian\ Gidea$^\dag$}
\address{Yeshiva University, Department of Mathematical Sciences, New York, NY 10016, USA }
\email{Marian.Gidea@yu.edu}

\author{Yitzchak Shmalo ${^\flat}$}
\address{Yeshiva University, Department of Mathematical Sciences, New York, NY 10016, USA }
\email{yitzchak.shmalo@mail.yu.edu }
\thanks{$^\dag$, $^\flat$ Research of M.G. and Y.S. was partially supported by NSF grant  DMS-0635607,  NSF grant DMS-1700154,  and by  the  Alfred P. Sloan Foundation grant G-2016-7320.}

\begin{abstract}
We present a combinatorial  approach to rigorously show the existence of fixed points, periodic orbits, and symbolic dynamics in discrete-time dynamical systems, as well as  to find numerical approximations of such objects. Our approach relies on the method  of `correctly aligned windows'. We subdivide the `windows' into cubical complexes, and we assign to the vertices of the cubes labels determined by the dynamics.
In this way we encode the dynamics information into a combinatorial structure.
We use a version of the Sperner Lemma saying that if the labeling satisfies certain conditions, then there exist fixed points/periodic orbits/orbits with prescribed itineraries.
Our arguments are elementary.
\end{abstract}


\maketitle

\section{Introduction}
Numerical investigations of  discrete-time dynamical systems often require the approximation of the phase space and of the underlying map via a fine grid. Henceforth, the dynamics information is encoded into a combinatorial structure.
From the computational point of view it is important that such combinatorial structure should be as simple as possible.
While simplicial structures appear to be more elegant, cubical  structures present many practical advantages, including  the possibility of using cartesian coordinates, simple numerical  representation of maps as  multivalued maps, and lower computational costs of  higher  dimension homologies.
See. e.g., \cite{Kaczynski2003}.

In this paper we develop a combinatorial topology-based approach  to detect fixed points, periodic orbits, and symbolic dynamics in discrete-time dynamical systems. The approach relies  on the method of `correctly aligned windows', also known as `covering relations'. This method goes back to the geometric ideas of Conley, Easton and McGehee \cite{Conley68,EastonMcGehee79,Easton81}, while more recent, topological versions of the method have been developed in \cite{GideaZ2004}. The method can be described concisely as follows. A `window' (also known as an `$h$-set') is a multi-dimensional rectangle, whose boundary consists of an `exit set' and an `entry set'. One window is correctly aligned with (or `covers') another window under the map if the image of the first window is going across  the  second window, with the exit set of the first window coming out through the exit set of the second window, and without the image of the first window intersecting the  entry set of the second window. There is an additional condition that the crossing of the windows should be topologically nontrivial, which can be expressed in terms of the Brouwer degree.   The main results about correctly aligned windows can be summarized as follow:
\begin{itemize}
\item[(i)] If a window is correctly aligned with itself, then there is a fixed point inside that window; \item[(ii)] For a finite sequence  of windows (with a circular ordering), if each window is correctly aligned with the next window in the sequence, then there exists a periodic orbit inside those windows; \item[(iii)] For an infinite sequence of windows, if each window is correctly aligned with the next window, then there exists an orbit inside those windows;  \item[(iv)] { For a finite sequence of pairwise disjoint windows, with the correct alignment of pairs of windows described by a $0-1$ transition matrix,  there exists an invariant set inside those windows on which the dynamics is semi-conjugate to a topological Markov chain associated to that matrix.}
\end{itemize}
In principle,  this method only yields  existential type of results on fixed points/perodic orbits/orbits with prescribed itineraries.

In this paper, we provide an algorithmic approach to verify the that the crossing of the windows is topologically non-trivial, and  to detect numerically, up to a desired level of precision, fixed points/perodic orbits/orbits with prescribed itineraries.   Our method is constructive, can be implemented numerically  quite easily, and does not require  the computation of algebraic topology-type invariants  (e.g., Conley index, homology, Brouwer degree).

The approach proposed in this paper is based on combinatorial topology, particularly on the classical Sperner  Lemma \cite{Sperner1928}.
We regard each window as a multi-dimensional cube, and we construct a cubical decomposition of it.
Then we assign labels to all vertices of the cubical decomposition. The  label of a vertex $x$  is given by the hyperoctant where the vector $f_\chi(x)-x$ lands, where $f_\chi$ is the map that defines the dynamical system expressed in certain coordinates.
A   cube  is called completely labeled if the intersection of the hyperoctants corresponding to its labels is the zero vector. We can also assign an index to the labeling of a   cube. This index turns out to be related to the Brouwer degree (see Proposition \ref{prop:Bekker}). Thus the index can be computed via a simple recursive formula (see \eqref{eqn:defn_index}).
A non-zero index is a sufficient condition, but not necessary,  for the  labeling to be complete.

To check that one window is correctly aligned with another, it is sufficient to check that the labels of the vertices of the cubical decomposition that lie on the boundary of one of the  windows satisfy certain explicit conditions, and that the above mentioned index is non-zero.

If a window is correctly aligned with itself, a version of the Sperner Lemma  shows the existence of at least one small cube in the cubical decomposition that has non-zero index, hence  completely labeled. There may also exists small cubes that are completely labeled that have zero index.
In this setting, a small cube  with non-zero index yields a true fixed point, while
a small cube that is completely labeled yields a numerical approximation of a fixed point.

Similarly, in the case of sequences of windows, small cubes inside those windows that have non-zero index yield
true periodic orbits/orbits with prescribed itineraries, while small cubes that are completely labeled yield
 numerical approximations of periodic orbits/orbits with prescribed itineraries.

Checking that a small cube has non-zero index can be done via a finite computation. Completely labeled small cubes can be searched using a Nerve Graph algorithm similar to \cite{Su2002}.

In Section \ref{sec:preliminaries}, we recall the classical Sperner Lemma and some generalizations. In Section \ref{sec:sperner_cubical} we provide a new  version of the Sperner Lemma, for cubical complexes, which will be used in the subsequent sections. In Section \ref{sec:windows} we provide sufficient conditions for correct alignment of windows, in terms of the labeling of vertices of a cubical decomposition. We also prove several results on detection  -- in terms of the non-zero index/completely labeled cubes --
of true/approximate fixed points, periodic orbits, and orbits with prescribed itineraries. In Section \ref{sec:application} we illustrate the above procedure by an example, in which we find periodic orbits for the H\'enon map.

\section{Preliminaries}\label{sec:preliminaries}
\subsection{Sperner's Lemma for simplices}\label{sec:sperner}
Consider an $n$-simplex $T$, and $T=\bigcup_{i\in I} T_i$, with $I$ finite, a simplicial decomposition of $T$.
A labeling of $T$ is a map $\phi:T\to\{1,2,\ldots,n+1\}$. In particular, each vertex $v$ of $T$ and of any of the $T_i$'s is assigned a unique label $\phi(v)\in \{1,2,\ldots,n+1\}$.

A simplex $T$ is said to be \emph{completely labeled} if its vertices are assigned all labels from $\{1,2,\ldots,n+1\}$.

The labeling $\phi$ is called a \emph{Sperner labeling} if every  point $p$ that lies on some face $S$ of $T$  is assigned one of the labels of the vertices of $S$.

In the case of a completely labeled simplex, the Sperner condition is equivalent to a  \emph{non-degenerate labeling} condition,  that no $(n-1)$-dimensional face of $T$ contains points of  $(n+1)$ or more different labels.

\begin{figure}
\includegraphics[width=0.5\textwidth]{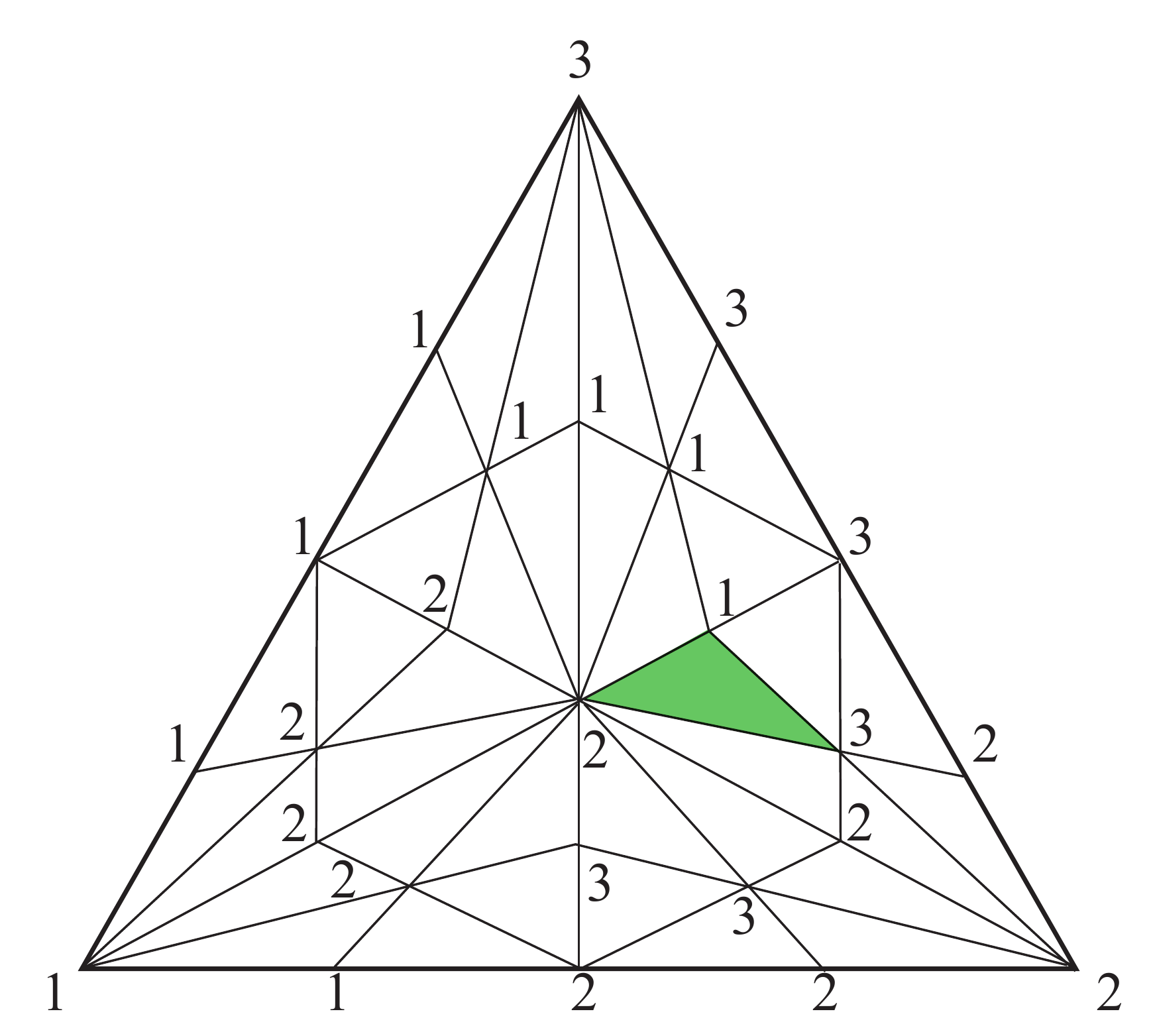}
\caption{Sperner labeling of a simplicial decomposition.}
\label{fig:sperner}
\end{figure}

In the most basic form the Sperner Lemma states the following:
\begin{thm}[Sperner Lemma \cite{Sperner1928}]
Given an $n$-simplex $T$ with a complete Sperner labeling,  a simplicial decomposition $T=\bigcup_{i} T_i$.
Then the number of completely labeled simplices $T_i$ is odd.
In particular,  there exists at least one simplex $T_i$ that is completely labeled.
\end{thm}

See Figure~\ref{fig:sperner}.
The Sperner Lemma can be used to derive an elementary proof of the Brouwer Fixed Point Theorem: any continuous map $f:B\to B$ from a  (homeomorphic copy of an) $n$-dimensional ball to itself has a fixed point. See, e.g.,
\cite{BurnsG2005}.

The Sperner Lemma can be used to numerically find fixed points, and it has been used extensively in numerical works,  see, e.g., \cite{Allgower2003}.

\subsection{Sperner's Lemma for polytopes}\label{sec:bekker}
We now describe a more general Sperner Lemma-type of result for polytopes, following \cite{Bekker1995}. A related  result can be found in \cite{Su2002}.

Let $P$ be a convex $n$-dimensional  polytope. We consider a labeling $\phi:P\to \{1,2,\ldots, n+1\}$ of $P$.
As in the simplex case, a polytope is said to be \emph{completely labeled} if  its vertices are assigned all labels from $\{1, 2, \ldots, n+1\}$.
Also, a labeling $\phi:P\to \{1,\ldots,n+1\}$ is said to be   \emph{non-degenerate} if no $(n-1)$-dimensional face of $P$ contains points which take $(n+1)$ or more different values.  
In the general case of a polytope, the non-degeneracy condition on the labeling is not equivalent to the Sperner condition.
In the sequel, we will assume that all labelings of the polytopes under consideration are non-degenerate.

In  applications, such as below, we often consider a polytope (cube) $P$ divided into a finite number of smaller polytopes (cubes) $P_i$, $i\in I$. In such  contexts, we only need to verify that the labels of the vertices of $P$ and the $P_i$'s lying on the faces of $P$ satisfy the \emph{non-degeneracy condition}. That is, we only need to verify such condition  on a finite number of points.

We  introduce some tools.

Consider the standard $n$-dimensional simplex $T=\textrm{conv}(0, e_1,e_2,\ldots ,e_n)\subset \mathbb{R}^n$, where we denote by $(e_1,e_2,\ldots ,e_{n})$ the standard basis of $\mathbb{R}^n$, and by $\textrm{conv}(\cdot)$ the convex hull of a set.
Let us declare the standard labeling of the vertices of $T$, given by  $\phi(0)=1$ and $\phi(e_i)=i+1$, $i=1,\ldots,n$,   as \emph{positively oriented}.  We also declare any labeling obtained by an even number of permutations of the labels of the vertices in the standard labeling also \emph{positively oriented}, and  any labeling obtained by an odd number of permutations  as \emph{negatively oriented}.

We define the \emph{oriented index} of a labeling $\phi$ of any oriented, convex polytope  $P$, recursively:
\begin{equation}\label{eqn:defn_index}
\textrm{ind}_P(\phi)=\left\{
                       \begin{array}{ll}
                         1, & \hbox{if $\dim(P)=0$ and $\phi(V(P))=\{1\}$;} \\[0.5em]
                         \pm 1, & \hbox{if $\dim(P)=1$ and $\phi(V(P))=\{1,2\}$,}\\ & {\textrm{according to its orientation};}\\[0.5em]
                         0, & \hbox{if $\dim(P)=k$ and $\phi(V(P))\neq\{1,2,\ldots,k+1\}$;} \\[0.5em]
                         \sum_{S\in\mathscr{F}_{k-1}(P)} \textrm{ind}_S(\phi), & \hbox{if $\dim(P)=k$ and $\phi(V(P))=\{1,2,\ldots,k+1\}$,}\\
& {\textrm{where } \textrm{ind}_S(\phi)  \textrm{ is counted with orientation}.}\\
                       \end{array}
                     \right.
\end{equation}
Above, $V(P)$ denotes the vertices of $P$, and  $\mathscr{F}_{k-1}( P)$ denotes the set of all $(k-1)$-faces of  $P$.
The orientation is taken into account in the following way. The orientation of a $k$-dimensional polytope $P$ induces an orientation of every $(k-1)$-dimensional face $S$ of $P$, and on all the lower dimensional faces of $S$.  In the definition of the index, only those $(k-1)$-dimensional faces that carry all labels  $\{1,2,\ldots,k\}$ count in the summation. The definition of the index yields a choice of sign $\pm 1$ to each  lower dimensional face carrying all labels which appears in the recursive definition.

\begin{rem}
By taking into account the orientation of the polytope, we get here a definition of the index slightly different from the one in \cite{Bekker1995}, which uses mod~2 summation in the above recursive definition. We note that the definition of the index in \cite{Bekker1995} is inconsistent with some of the results later in that paper, e.g., with Proposition \ref{prop:Bekker} quoted below.
\end{rem}

Notice that the index of a labeling $\phi$ of a polytope  defined above only depends on the values of $\phi$ at the vertices $V(P)$ of $P$.

The following result is immediate:
\begin{lem}[\cite{Bekker1995}]\label{lem:simplex}
If $T$ is a simplex and $\phi:T\to \{1,\ldots,n+1\}$ is a non-degenerate labeling, then
\[\ind_T(\phi)=\left\{
                       \begin{array}{ll}
                         \pm 1 , & \hbox{if $T$ is completely labeled;} \\
                        0, & \hbox{otherwise.}
                       \end{array}
                     \right.\]
\end{lem}

Given an $n$-dimensional, oriented, convex polytope  $P$, a   labeling  $\phi:P\to\{1,\ldots,n+1\}$, and the standard $n$-simplex $T$ of vertices $a_1,\ldots,a_{n+1}$, a \emph{realization} of $\phi$ is a continuous map $\Phi: P \to T$, satisfying the following condition:
\begin{itemize}
\item[(i)] If $v$ is a vertex of $P$ then $\Phi(v)=a_{\phi(v)}$, i.e., $\Phi(v)$ is the vertex $a_i$ of  $T$ with the index $i$ equal to the label of $v$;
\item[(ii)] If $S$ is face of $P$  with vertices $v_1,\ldots,v_k$,  then $\Phi(S)\subset \textrm{conv}(a_{\phi(v_1)},\ldots, a_{\phi(v_k)})$.
\end{itemize}
Informally, a realization of $P$ is a continuous mapping of $P$ onto $T$ that `wraps' $\partial P$ around $\partial T$, such that the labels of the vertices of $P$ match with the indices $i$ of the vertices of $T$. Such $\Phi$ is in general non-injective.

In the sequel we denote by $\deg(\cdot)$ the oriented Brouwer degree of a continuous function (see, e.g., \cite{BurnsG2005}).
Recall that for a smooth, boundary preserving map $\Phi:(M,\partial M)\to (N,\partial N)$ between two oriented $n$-dimensional manifolds with boundary, $\deg(\Phi)=({\int_{ M} \Phi^*\eta})/({\int _{ N} \eta})$ where $\eta$ is a volume form on $N$; the definition is independent of the volume form. Let $\omega$ be such that  $d\omega=\eta$. By Stokes' theorem and the fact that $d\Phi^*\omega=\Phi^*d\omega=\Phi^*\eta$, we have $\deg(\Phi)=({\int_{\partial M} \Phi^*\omega})/({\int _{\partial N} \omega})$. This implies the following property of the degree: $\deg(\Phi)=\deg(\partial \Phi)$, where  $\partial \Phi:\partial  M\to\partial N$ is the map induced by $\Phi$ on the boundaries.   Equivalently, the degree of the map $\Phi$ can be   defined as the signed number of preimages $\Phi^{-1}(p)=\{q_1,\ldots,q_k\}$ of a regular value  $p$ of the map $\Phi$, where each point $q_i$ is counted with a sign $\pm 1$ depending on whether  $d\Phi_{q_i}:T_{q_i}M\to T_{p}N$ is orientation preserving or orientation reversing. That is, $\deg(\Phi)=\sum_{q\in \Phi^{-1}(p)}\textrm{sign} (\det (d\Phi_{q}))$, where $p\in  N\setminus \partial N$ is a regular value of $\Phi$. The definition of the Brouwer degree extends via homotopy to continuous maps.
\newpage
\begin{prop}[\cite{Bekker1995}]\label{prop:Bekker}  Let $P$  be a convex  $n$-dimensional polytope.
\begin{itemize}
\item[(i)] Any non-degenerate  labeling $\phi$ of $P$ admits a realization $\Phi$;
\item[(ii)] Any two realizations of the same labeling are homotopic as maps of pairs $(P,\partial P)\mapsto (T,\partial T)$;
\item [(iii)] The index $\ind_P(\phi)$ of the labeling $\phi$ is equal to the degree $\deg(\Phi)$ of any realization $\Phi$ of $\phi$, up to a sign
\[\ind_P(\phi)=\pm \deg(\Phi).\]
\item[(iv)] If $\ind_P(\phi)\neq 0$ then $P$ is completely labeled.
\end{itemize}
\end{prop}

Let us consider  a convex, $n$-dimensional polytope $P$ that is subdivided into finitely many  $n$-dimensional polytopes $\{P_i\}_{i\in I}$, with $I$ finite, such that $P=\bigcup_{i\in I}P_i$, and for $i\neq j$, $\inte (P_i)\cap \inte (P_j)=\emptyset$ and $P_i\cap P_j$ is either empty or a face of both $P_i$ and $P_j$. In particular, each $k$-dimensional face of $P$ is   the union of finitely many $k$-dimensional faces of $P_i$'s.

For the following result from  \cite{Bekker1995} we provide an alternative proof.

\begin{lem}[\cite{Bekker1995}]\label{lem:sumindices}
If the labeling $\phi$ is non-degenerate, then
\[\ind_P(\phi)=\sum_i \ind_{P_i} (\phi).\]
\end{lem}

\begin{proof}
The proof follows by induction on the dimension $n$ of the polytope. When $n=1$ the identity is immediate.
For the induction step,  we use  \eqref{eqn:defn_index}. Note that each $(n-1)$-dimensional face $S$ of $P$ is the union of $(n-1)$-dimensional faces of $P_i$'s. Each $(n-1)$-dimensional face of a $P_i$ that is not lying on a $(n-1)$-dimensional face of $P$ is shared by two polytopes $P_i$ and $P_j$, and so it is counted twice with  opposite orientations.
Thus, the sum of the indices of the $(n-1)$-dimensional faces of the $P_i$'s reduces to the sum of the the indices of the $(n-1)$-dimensional faces of the $P_i$'s that lie on $(n-1)$-dimensional faces of $P$. The fact that the index of each $(n-1)$-dimensional face $S^j$ of $P$ is the sum of the indices of the $(n-1)$-dimensional faces $S^j_i$ of the $P_i$'s that lie on $S^j$ follows from the induction hypothesis.
In summary, we have:
\begin{equation*}\begin{split}\sum_i\ind_{P_i}(\phi)&=\sum_i\sum_{S'\in\mathscr{F}_{n-1}(P_i)} \textrm{ind}_{S'}(\phi)\\&=\sum_i\sum_{S^j_i\in\mathscr{F}_{n-1}(P_i)\cap \mathscr{F}_{n-1}(P)} \textrm{ind}_{S^j_i}(\phi)\\&=\sum_{S^j\in\mathscr{F}_{n-1}(P)} \textrm{ind}_{S^j}(\phi)\\&=\ind_{P}(\phi).
\end{split}\end{equation*}
\end{proof}


The following is a generalization of the Sperner Lemma from  \cite{Bekker1995}.

\begin{thm}[\cite{Bekker1995}]\label{thm:bekker2}
Assume that $P$ is an $n$-dimensional polytope,   $P=\bigcup_{i\in I}P_i$ is a decomposition of $P$ into polytopes as above, and $\phi:P\to\{1,\ldots,n+1\}$ is a {\emph{non-degenerate labeling}}.   If $\ind_P(\phi)\neq 0$, then there exists a polytope $P_i$ such that $\ind_{P_i}(\phi)\neq 0$; in particular, $P_i$ is completely labeled.
\end{thm}
\begin{proof}
Follows immediately from Lemma \ref{lem:sumindices} and Proposition \ref{prop:Bekker}.
\end{proof}

\begin{cor}\label{cor:Bekker}
Assume that $P$ is an $n$-dimensional polytope, and $P=\bigcup_{i\in I}T_i$ is a simplicial decomposition of $P$. If a { non-degenerate} labeling $\phi$   satisfies $\ind_P(\phi)\neq 0$, then there exists a simplex $T_i$ that is completely labeled.
\end{cor}

\begin{proof}
Follows from Lemma \ref{lem:simplex}.
\end{proof}

Note that in Corollary \ref{cor:Bekker} the assumption that $P$ is completely labeled alone is not sufficient to ensure that there exists a completely labeled simplex in the decomposition; the condition that $\ind_P(\phi)\neq 0$  is necessary. See Example \ref{ex}, (ii), (iii).

\begin{ex}\label{ex}

(i) Consider the polygon $P$, the simplicial decomposition,   and the labeling shown in Fig. \ref{hexagon}-(a). We have  $\ind_P(\phi)=2$; there exists a completely labeled triangle.

(ii) Consider the polygon $P'$, the simplicial decomposition,    and the labeling shown in Fig. \ref{hexagon}-(b).   We have $\ind_{P'}(\phi)=0$; there is no completely labeled simplex.

(iii) Consider the polyhedron $P''$, the simplicial decomposition,    and the labeling shown in Fig. \ref{hexagon}-(c). We have  $\ind_{P''}(\phi)=0$;  there is no completely labeled simplex.
\begin{figure}
$\begin{array}{ccc}
\includegraphics[width=0.25\textwidth]{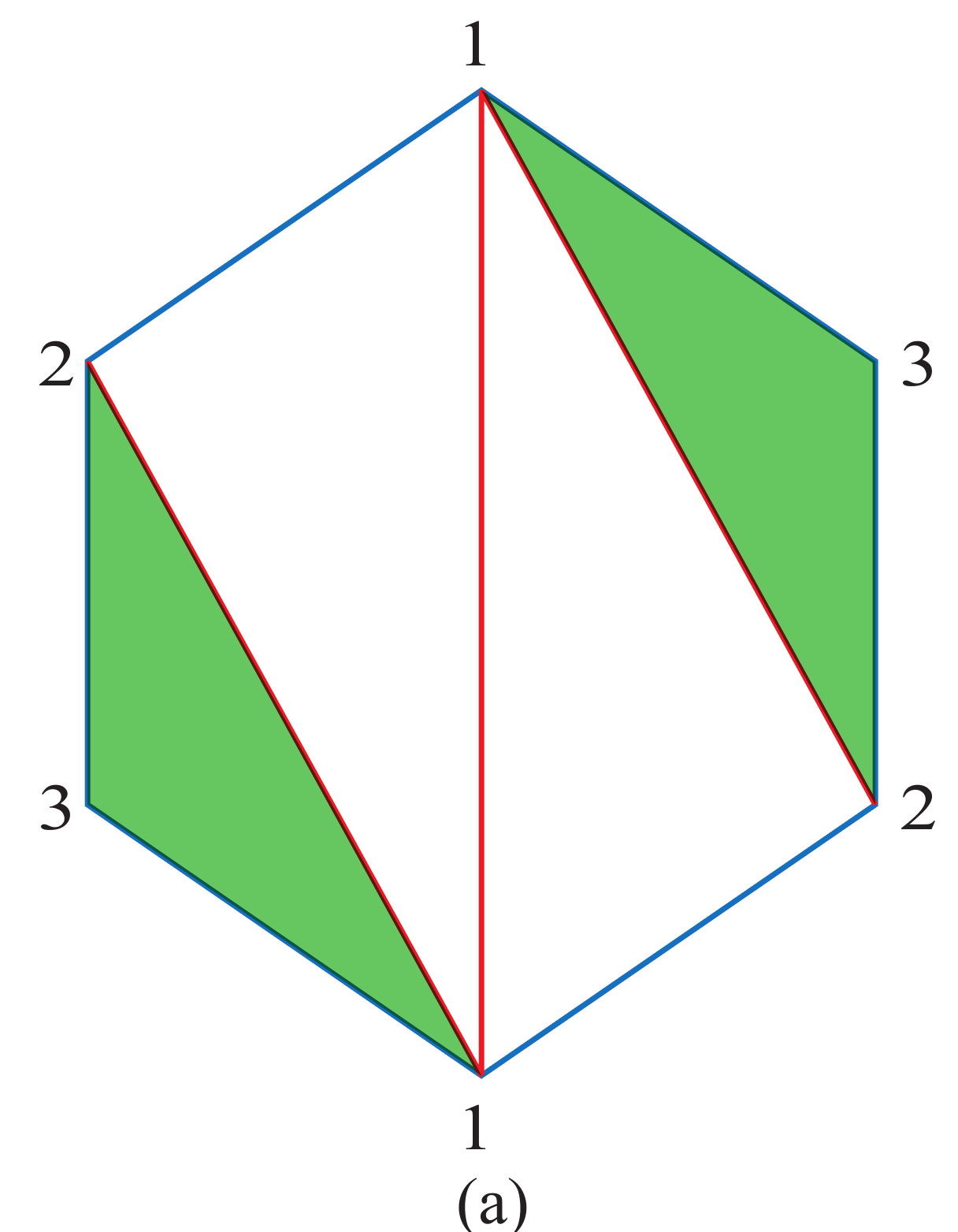}&
\includegraphics[width=0.25\textwidth]{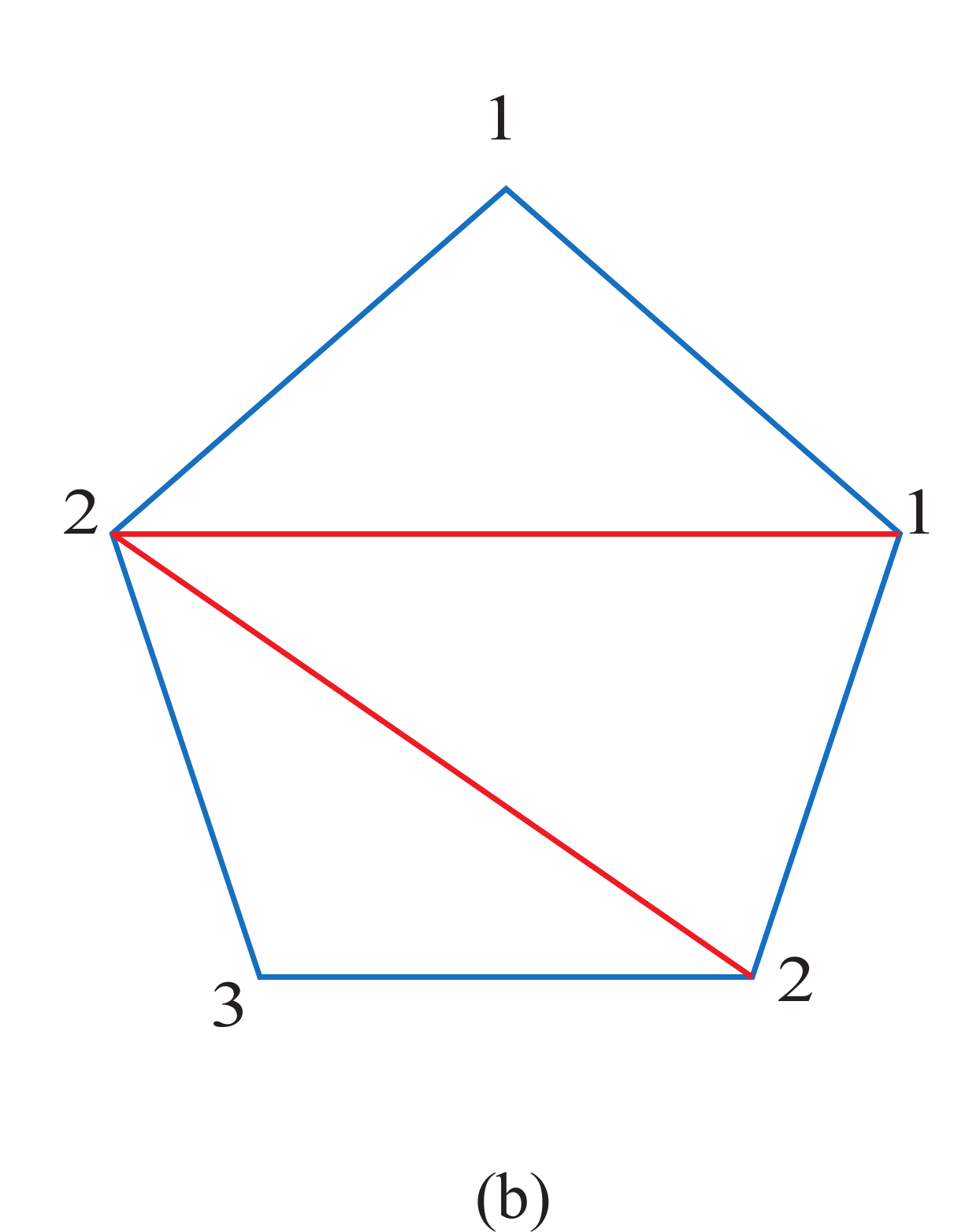}&
\includegraphics[width=0.25\textwidth]{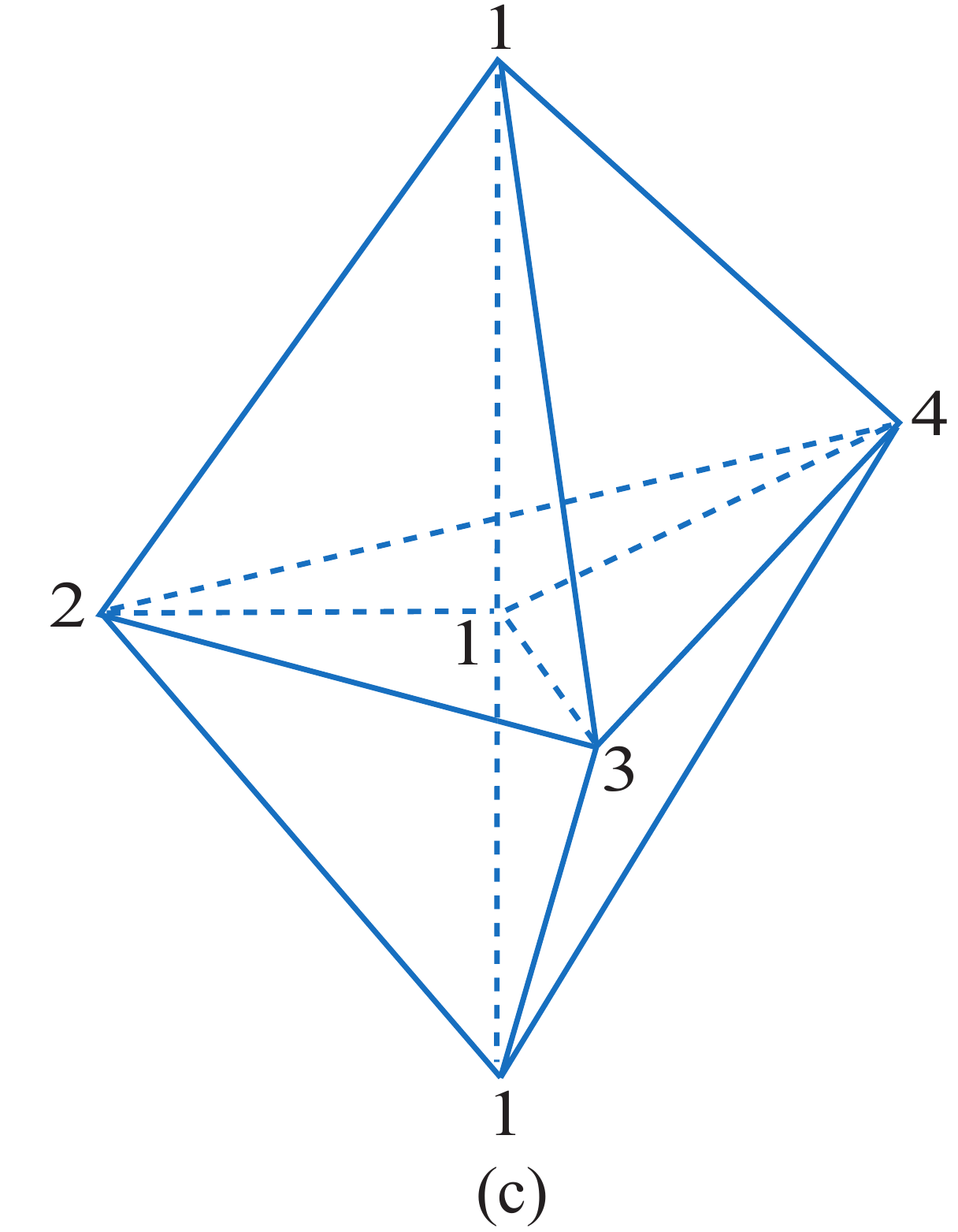}
\end{array}$
\caption{Examples of simplicial decomposition of polytopes and of labelings.}
\label{hexagon}
\end{figure}
\end{ex}

\section{Sperner's Lemma for cubical complexes}\label{sec:sperner_cubical}

In this section we present a new version of the Sperner  Lemma for cubical decompositions. The main difference from the previous sections will be the labeling. For an $n$-dimensional cube and a corresponding cubical decomposition, it will be convenient in Section \ref{sec:windows} to use $2^n$ labels that are $n$-dimensional vectors with coefficients $\pm 1$; whereas, in Section \ref{sec:bekker}, for an $n$-dimensional polytope we have used only $(n+1)$ labels. Hence,  we will have to re-define what a complete labeling means in terms of the new labeling convention, and relate with the old labeling convention.

Consider  vector-labels $\ell\in \{(\pm 1,\ldots, \pm 1)\}=\mathscr{Z}^n$, where we denote $\mathscr{Z}=\{\pm 1\}$.
Each label $\ell=(\pm 1,\ldots, \pm 1)$ corresponds to  a hyperoctant
\[\mathscr{O}_\ell=\{(x_1,x_2,\ldots, x_n)\in  \R^n\,|\, \forall i,\,\ell_i x_i\geq 0 \}.\]
Note that $\bigcup_{\ell\in\mathscr{Z}^n} \mathscr{O}_\ell=\R^n$.

We call a collection of   labels $\{\ell_1,\ell_2,\ldots,\ell_{n+1}\}$  \emph{complete} if $\mathscr{O}_{\ell_1}\cap \mathscr{O}_{\ell_2}\cap\ldots\cap \mathscr{O}_{\ell_{n+1}}=\{{\bf 0}\}$.
Equivalently, $\{\ell_1,\ell_2,\ldots,\ell_{n+1}\}$ is complete if for each coordinate index $i\in\{1,\ldots,n\}$, there exists a pair of labels $\ell_{j}, \ell_{k}$ such that the $i$-th coordinates of $\ell_{j}$ and $\ell_{k}$ have opposite signs, that is, $\pi_i(\ell_j)=-\pi_i(\ell_k)$.

A labeling $\phi$   is said to be \emph{non-degenerate} if no face of $P$  carries a complete set of labels.

A convex  polyhedral cone is a convex cone  in $\R^n$ bounded by a finite collection of hyperplanes of the form $x_k=0$; alternatively, it can be characterized  as an intersection of finitely many half-spaces (e.g., spaces of the form $\{x\in\R^n\,|\,x_k\geq 0\}$ for some $k$); see \cite{Stanley2011}.

A \emph{special  convex  polyhedral cone partition} of $\R^n$ is a collection of $(n+1)$ convex  polyhedral cones $\mathscr{N}_j$, $j=1,\ldots,n+1$, satisfying the following properties:
\begin{itemize}
\item[(a)] $\bigcup _j \mathscr{N}_j=\R^n$;
\item[(b)] $\inte({\mathscr{N}_j})\cap \inte(\mathscr{N}_l)=\emptyset$ for $j\neq l$;
\item[(c)] $\bigcap_j \mathscr{N}_j=\{\bf{0}\}$.
\end{itemize}

\begin{lem}\label{lem:partition}
(i) Given a special  convex  polyhedral cone partition   $\mathscr{N}_1, \ldots, \mathscr{N}_{n+1}$ of   $\R^n$. Any set of $(n+1)$ hyperoctants $\mathscr{O}_{\ell_1},\ldots,\mathscr{O}_{\ell_{n+1}}$, with the property that each $\mathscr{O}_\ell$ is contained in exactly one $\mathscr{N}_j$, and no two $\mathscr{O}_\ell$'s are contained in the same
$\mathscr{N}_i$, satisfies $\mathscr{O}_{\ell_1}\cap \mathscr{O}_{\ell_2}\cap\ldots\cap \mathscr{O}_{\ell_{n+1}}=\{\bf 0\}$.

(ii) Given  a set of $(n+1)$ hyperoctants $\mathscr{O}_{\ell_1},\ldots,\mathscr{O}_{\ell_{n+1}}$ with $\mathscr{O}_{\ell_1}\cap \mathscr{O}_{\ell_2}\cap\ldots\cap \mathscr{O}_{\ell_{n+1}}=\{\bf 0\}$. There exists a special  convex  polyhedral cone partition   $\mathscr{N}_1, \ldots, \mathscr{N}_{n+1}$ of   $\R^n$,  such that each $\mathscr{O}_{\ell_i}$ is contained exactly in one $\mathscr{N}_j$.
\end{lem}
\begin{proof}
(i) We have  $\mathscr{O}_{\ell_1}\cap \mathscr{O}_{\ell_2}\cap\ldots\cap \mathscr{O}_{\ell_{n+1}}\subseteq \mathscr{N}_1\cap \ldots \cap \mathscr{N}_{n+1}=\{\bf 0\}$.

(ii) Consider a  set of  hyperoctants $\mathscr{O}_{\ell_1},\ldots,\mathscr{O}_{\ell_{n+1}}$ with $\mathscr{O}_{\ell_1}\cap \mathscr{O}_{\ell_2}\cap\ldots\cap \mathscr{O}_{\ell_{n+1}}=\{\bf 0\}$. Each hyperplane will separate the set of hyperoctants into two non-empty collections on each side of the hyperplane (since having all hyperoctants on the same side of a hyperplane would imply that their intersections is more than the zero vector).  Thus, cutting $\R^n$ by the $n$ hyper-planes will imply that each pair of hyperoctants is on opposite sides of some hyperplane.
First,  select $\mathscr{N}_1$ as the largest  intersection of half-spaces in $\R^n$   (i.e., an intersection by the minimum number of half-spaces) that contains only one hyperoctant $\mathscr{O}_{\ell_{i_1}}$. Then select $\mathscr{N}_2$ as the largest  intersection set of half-spaces in $\R^n\setminus \mathscr{N}_1$  that contains only one hyperoctant $\mathscr{O}_{\ell_{i_2}}$, with $i_1\neq i_1$. Continue this procedure up to the last hyperoctant $\mathscr{O}_{\ell_{i_{n+1}}}$ in the given collection, which will provide $\mathscr{N}_{n+1}$.
\end{proof}

The following example shows some simple changes of labels:
\begin{ex}\label{ex:relabeling} Consider the following   special partition into convex polyhedral cones:
\begin{equation*}\begin{split}
\mathscr{N}_1&=\{x\in\R^n\,|\,x_1\geq 0\},\\
\mathscr{N}_2&=\{x\in\R^n\,|\,x_1\leq 0, x_2\geq 0\},\\
&\cdots \\
\mathscr{N}_{n}&=\{x\in\R^n\,|\,x_1\leq 0, \ldots,x_{n-1}\leq 0,x_{n}\geq 0\},\\
 \mathscr{N}_{n+1}&=\{x\in\R^n\,|\,x_1\leq 0, \ldots,x_{n-1}\leq 0,x_{n}\leq 0\}.
\end{split}\end{equation*}

The corresponding  change of labels $\psi$  from   vector-labels $\ell\in \mathscr{Z}^n$ to   labels $j \in \{1,2,\ldots,
n + 1\}$ is given by
\begin{equation*}\begin{split}\psi(1,\pm 1, \ldots, \pm 1) &= 1,\\ \psi(-1,1, \pm 1, \ldots, \pm 1) &= 2,\\
&\cdots \\
\psi(-1, \ldots, -1, 1,\pm 1, \ldots,\pm 1) &= \min\{j  |\, \ell_j = 1\},\\&\cdots \\ \psi(-1,-1, \ldots, -1) &= n+1.\end{split}\end{equation*}
\end{ex}

Consider a labeling $\phi:P\to \mathscr{Z}^n$ of an $n$-dimensional polytope $P$.
Let $\psi:\mathscr{Z}^n\to\{1,\ldots,n+1\}$ be a change of labels. We can define a re-labeling $\psi\circ\phi:P\to \{1,\ldots,n+1\}$. This re-labeling is as  in Section \ref{sec:bekker},   hence all the results from that section can be applied in this context. In particular, any re-labeling $\psi\circ\phi: P\to \{1,\ldots,n+1\}$ admits a realization.

\begin{lem}\label{lem:index_cubical} Let $P$ be a polytope,  $\phi:P\to\mathscr{Z}^n$  a labeling,  $\psi_1,\psi_2:\mathscr{Z}^n\to\{1,\ldots,n+1\}$   two changes of labels, and  $\psi_1\circ\phi,\psi_2\circ\phi:P\to\{1,\ldots,n+1\}$  the corresponding re-labelings.  Then \[\ind_{P} (\psi_1\circ\phi)=\pm\ind_{P} (\psi_2\circ\phi).\] \end{lem}

\begin{proof} Let  $\Phi_i:P\to T$ be a realization of $\psi_i\circ\phi$, $i=1,2$.
By  Proposition~\ref{prop:Bekker}, $\ind_{P}(\psi_i\circ\phi)=\pm\deg(\Phi_i)$, for $i=1,2$.
Each change of labels $\psi_i$ corresponds to a special partition into convex polyhedral cones $\{\mathscr{N}^i_j\}_j$.
Either partition can be obtained from the other via a composition of a rotation, a reflection, and a homotopy.
These transformation preserve the Brouwer degree up to a sign.
Hence $\deg(\Phi_1)=\pm\deg(\Phi_2)$,  thus $\ind_{P} (\psi_1\circ\phi)=\pm\ind_{P} (\psi_2\circ\phi)$.
\end{proof}

By Lemma \ref{lem:index_cubical}, we can define the index of a labeling $\phi:P\to\mathscr{Z}^n$, up to a sign, by $\ind_{P}(\phi)=\pm\ind_{P} (\psi \circ\phi)$, where   $\psi$ is a change of labels.

Now, let us consider $C$  and a labeling  $\phi:C\to\mathscr{Z}^n$ of $C$.   Let $\{C_i\}_i$   be a cubical decomposition of $C$.
Then $C$ can be regarded as a polytope  by appending to the vertices of $C$ all the
vertices of the the $C_i$'s lying on the faces of $C$, as well as appending to the faces of
$C$ all the faces of the $C_i$'s lying on the faces of $C$.
We denote this polytope by $\tilde C$.
For $\tilde C$ we have a polytope (cubical) decomposition $\tilde C=\bigcup_i C_i$. The labeling $\phi:C\to\mathscr{Z}^n$ of $C$ can be viewed as a labeling $\phi:\tilde C\to\mathscr{Z}^n$ of $\tilde C$.

The following is a version of Sperner's Lemma for cubical decompositions.

\begin{thm}[Sperner's Lemma]\label{thm:sperner_cubical} Let $C=\bigcup_i C_i$ be  a cube together with a cubical decomposition, $\tilde C$ the corresponding polytope, and $\phi:\tilde C\to\mathscr{Z}^n$   a  { non-degenerate} labeling of $\tilde C$.
If $\ind_{\tilde C}(\phi)\neq 0$ then there exists at least one cube $C_i$ such that $\ind_{C_i}(\phi)\neq 0$, hence completely labeled relative to $\phi$.
\end{thm}

\begin{proof} Consider a re-labeling $\psi\circ\phi$ of $\tilde C$. By definition, we have $\ind_{\tilde C}(\phi)=\pm\ind_{\tilde C}(\psi\circ\phi)  \neq 0$. By Theorem \ref{thm:bekker2}, there exists a  cube $C_i$ such that $\ind_{C_i}(\psi\circ\phi)\neq 0$, and so $C_i$ is completely labeled relative to $\psi\circ\phi$. Hence $\ind_{C_i}(\phi)\neq 0$. Lemma \ref{lem:partition} says that a labeling is  complete   relative to $\psi\circ\phi$ if and only if  it is completely labeled relative to~$\phi$.
\end{proof}

\begin{cor}\label{cor:sperner_cubical_1}
Let $C=\bigcup_i C_i$ be  a cube together with a cubical decomposition, and $\phi:C\to\mathscr{Z}^n$   a  { non-degenerate} labeling of $C$.
If $\ind_{C}(\phi)\neq 0$ then there exists at least one cube $C_i$ such that $\ind_{C_i}(\phi)\neq 0$, hence completely labeled relative to $\phi$.
\end{cor}
\begin{proof}
If the labeling $\phi$ of $C$ is { non-degenerate}, and $\tilde C$ is the polytope obtained from the cubical decomposition $C=\bigcup_iC_i$, it follows that $\ind_{\tilde C}(\phi)=\ind_{C}(\phi)\neq 0$. Thus Theorem \ref{thm:sperner_cubical} applies.
\end{proof}

These last two results are the most important for Section \ref{sec:windows}.
In a  nutshell, they say that given a cubical complex with a non-degenerate, non-zero index labeling,
any finer decomposition of the complex into smaller cubes will always have a small cube with non-zero index.

\section{Correctly aligned windows and detection of fixed points/periodic orbits/orbits with prescribed itineraries}\label{sec:windows}
In this section we present the definitions of windows and of correct alignment following \cite{GideaZ2004}, with a few minor modifications. Then we associate some labeling to the windows, and characterize correct alignment in terms of that labeling. We also use the labeling to find numerical approximations of fixed points, periodic orbits, and orbits with prescribed itineraries.

\subsection{Approximate fixed points and periodic orbits}
Given $\delta>0$ and a map $g:\R^n\to \R^n$, we call a point $z=(x_1,\ldots,x_n)\in \R^n$ a \emph{$\delta$-approximate fixed point}  of $g$ if $\|g(z)-z\|_\infty<\delta$,  where $\|z\|_\infty=\max_{i=1,\ldots,n}|x_i|$.

We remark here that   there may be no `true' fixed point near a  $\delta$-approximate fixed point, that is,  $\delta$-approximate fixed points can be  `fake' fixed points. Obviously, any `true' fixed point is a \emph{$\delta$-approximate fixed point} for any $\delta>0$.

Similarly, a finite collection of points $z_1,\ldots, z_k$ is a \emph{$\delta$-approximate periodic orbit}  of $g$
if $\|g(z_j)-z_{j+1}\|_\infty<\delta$, for $j=1,\ldots,k$, and $\|g(z_k)-z_{1}\|_\infty<\delta$.

As in the case of fixed points, there may be no `true' periodic orbits near a  $\delta$-approximate periodic orbit.

\subsection{Windows}
We consider  a discrete dynamical system given by a homeomorphism $f:\R^n\to\R^n$.

We define an equivalence relation on the set of homeomorphisms $\chi:\R^n\to\R^n$ by setting $\chi_1\sim\chi_2$ if there exists an open neighborhood $U$ of $[0,1]^n$ in $\R^n$ such that $(\chi_1)_{\mid U}=(\chi_2)_{\mid U}$.  We will use the same notation for an  equivalence class as for a representative of that class.

\begin{defn} A window in $\R^n$ consists of:\begin{itemize} \item[(i)] a homeomorphic copy $D$ of a multidimensional rectangle $[0,1]^n$,
\item[(ii)] an equivalence class of  homeomorphisms $\chi_{D}:\R^n\to  \R^n$ with $\chi_D([0,1]^n)=D$,
\item[(iii)] a choice of stable- and unstable-like dimensions,  $n_s,n_u\geq 0$, respectively,  with $n_s+n_u=n$,
\item[(iv)] a choice of an `exit set' $D^-\subset \bd (D)$ and of an `entry set' $D^+\subset \bd(D)$, given by
\begin{equation*}
\begin{split}
D^-&=\chi_D\left(\partial[0,1]^{n_u} \times[0,1]^{n_s}\right),\\
D^+&=\chi_D\left([0,1]^{n_u}\times\partial[0,1]^{n_s} \right).
\end{split}
\end{equation*}
\end{itemize}
\end{defn}

We will write $[0,1]^n=[0,1]^{n_u}\times [0,1]^{n_s}$.
Given $D_1,D_2$  two windows in $\R^n$, and $\chi_{D_1},\chi_{D_2}$   two representatives of the corresponding equivalence classes of homeomorphisms, we
denote by  $f_{\chi_{D_1},\chi_{D_2}}:\R^n\to\R^n$ the homeomorphism $f_{\chi_{D_1},\chi_{D_2}}:=\chi_{D_2}^{-1}\circ f \circ \chi_{D_1}$. When there is no risk of ambiguity, we use the simplified notation $f_\chi:= f_{\chi_{D_1},\chi_{D_2}}$.

Denote by $\Upsilon :=\{(x,y)\in \R^{n_u}\times \R^{n_s}\,|\,x\not\in [0,1]^{n_u}\}$.
Given two windows $D_1$ and $D_2$ such that $f_\chi ([0,1]^n)\subseteq  \Upsilon  \cup \left( [0,1]^{n_u}\times (0,1)^{n_s}\right)$,  there exists another homeomorphism $\chi'_{D_2}$ from the equivalence class of homeomorphisms associated to  $D_2$, such that $f_{\chi'} ([0,1]^n )\subset \R^{n_u}\times (0,1)^{n_s}$, where $f_{\chi'}:={\chi'_{D_2}}^{-1}\circ f \circ  \chi_{D_1}$. Note that changing $\chi_{D_2}$ with $\chi'_{D_2}$ has no effect on the windows $D_1$ and $D_2$.

\subsection{Correctly aligned windows in two-dimensions} \label{sec:2D_windows}
We first give a definition of correct alignment of windows in the  $2$-dimensional case, that is,  $n=2$.

\begin{defn}\label{defn:win2d} We say that the window $D_1$ is correctly aligned with $D_2$ under $f$ if there exist  corresponding homeomorphisms $\chi_{D_1},\chi_{D_2}$ with the following properties:
\begin{itemize}
\item[(i)] Case $n_u=1$, $n_s=1$:
\begin{itemize}\item[(i.a)] $f_{\chi} ([0,1]^2 )\subset \R\times(0,1)$;
\item [(i.b)]  $f_{\chi} (\{0\}\times [0,1]  ) \subset
(-\infty,0)\times \R$ and $f_{\chi} (\{1\}\times [0,1]  )\subset
(1, +\infty) \times \R$, or, $f_{\chi} (\{0\}\times [0,1]  )\subset
(1, +\infty) \times  \R$ and $f_{\chi} (\{1\}\times [0,1]  )\subset
(-\infty, 0) \times  \R$;
 \end{itemize}
\item[(ii)]
Case  $n_u=0$, $n_s=2$:
  $f_{\chi}( [0,1]^2 )\subset
(0,1)^2$;
\item[(iii)] Case  $n_u=2$, $n_s=0$:
 $f^{-1}_{\chi} ([0,1]^2) \subset (0,1)^2$.
\end{itemize}
\end{defn}

\begin{rem}
As noted above, instead of condition (i.a) we can require the more general condition that $f_{\chi} ([0,1]^2 )\subset \Upsilon  \cup \left( [0,1]^{n_u}\times (0,1)^{n_s}\right)$.
\end{rem}

Let $g:\R^2\to\R^2$ be a homeomorphism.
Let $(x_1,x_2)$ denote the coordinates of a point $z\in [0,1]^2$ and $(x'_1,x'_2)$ the coordinates  of $g(z)$. Let $\Delta z:=(\Delta x_1,\Delta x_2)=(x_1'-x_1,x'_2-x_2)$.
Denote the quadrants \begin{equation*}\begin{split}\mathscr{O}_{(1,1)}&=\{(x_1,x_2)\,|\,  x_1\geq 0\textrm{ and } x_2\geq 0\},\\  \mathscr{O}_{(-1,1)}&=\{(x_1,x_2)\,|\, x_1\leq 0\textrm{ and } x_2\geq 0\},\\   \mathscr{O}_{(-1,-1)}&=\{(x_1,x_2)\,|\,  x_1\leq 0\textrm{ and } x_2\leq 0\}, \\   \mathscr{O}_{(1,-1)}&=\{(x_1,x_2)\,|\,  x_1\geq 0\textrm{ and } x_2\leq 0\}.\end{split}\end{equation*}
These are closed sets which cover $\R^2$, and $\mathscr{O}_{\ell_1}\cap \mathscr{O}_{\ell_2}\cap \mathscr{O}_{\ell_3}=\{(0,0)\}$ whenever $\ell_1,\ell_2,\ell_3$ are all different. With respect to the mapping $g$, to each point $z\in \R^2$ we assign a vector label $(\pm 1,\pm 1)\in\mathscr{Z}^2$ according to the following:

\begin{description}
\item[Condition O]{$ $}

\begin{itemize}\item if $\Delta z\in\inte (\mathscr{O}_\ell)$ then we assign to $z$ the label  $\ell$;
\item if  $\Delta z\in \inte(\mathscr{O}_{\ell_1}\cap \mathscr{O}_{\ell_2})$  then we  assign to $z$ either label $\ell_1$ or $\ell_2$;
\item if $\Delta z=0$,  we assign to $z$ either label $\ell$.
\end{itemize}
\end{description}

A square $C\subset \R^2$ labeled according to \emph{Condition O} is said to be \emph{completely labeled} if its   vertices contain three different labels $\ell_1,\ell_2,\ell_3$,  in which case $\mathscr{O}_{\ell_1}\cap \mathscr{O}_{\ell_2}\cap \mathscr{O}_{\ell_3}=0$.
Note that if $\Delta z\in \mathscr{O}_{\ell_1}\cap \mathscr{O}_{\ell_2}\cap \mathscr{O}_{\ell_3}$ with  $\ell_1,\ell_2,\ell_3$  mutually distinct, then $\Delta z=0$ and so $(x'_1,x'_2)=(x_1,x_2)$.

Suppose that $D_1$ is correctly aligned with $D_2$ under $f$. For the mapping $f_\chi:[0,1]^2\to \R^2$ we assign a labeling $\phi:[0,1]^2\to \mathscr{Z}^2$  as per \emph{Condition O}. Denote the vertices of $[0,1]^2$ as follows: $A=(1,0)$, $B=(0,0)$, $C=(0,1)$, $D=(1,1)$.   From the definition of correct alignment, we infer the following    labeling of the vertices and the edges  of $[0,1]^2$:
\begin{itemize}
\item[(i)] Case $n_u=1$, $n_s=1$. Using the labeling associated to $f_\chi$ yields:
\begin{itemize}
\item [(i.a)]
$A\to (1,1)$, $B\to (-1,1)$,   $C\to (-1,-1)$, $D\to (1,-1)$, or

    $B\to (1,1)$, $A\to (-1,1)$,   $D\to (-1,-1)$, $C\to (1,-1)$,
\item [(i.b)]   $AB\to (1,1)$  or $(-1,1)$,   $CD\to  (-1,-1)$ or $(1,-1)$;
\item [(i.c)]   $BC\to  (-1,1)$ or (-1,-1),  $AD\to  (1,1)$ or $(1,-1)$, or

     $BC\to   (1,1)$ or $(1,-1)$,  $AD\to  (-1,1)$ or $(-1,-1)$;
 \end{itemize}
\item[(ii)] Case  $n_u=0$, $n_s=2$. Using the labeling associated to $f_\chi$ yields:
\begin{itemize}\item[(ii.a)]
$A\to (-1,1)$, $B\to (1,1)$,   $C\to (1,-1)$, $D\to (-1,-1)$,
\item[(ii.b)]    $AB\to   (1,1)$ or $(-1,1)$,   $AD\to   (-1,1)$ or $(-1,-1)$,   $CD\to (-1,-1)$ or
$(1,-1)$,   $BC\to   (1,1)$ or $(1,-1)$;
 \end{itemize}
\item[(iii)] Case  $n_u=2$, $n_s=0$. Using the labeling associated to $f^{-1}_\chi$ yields the same labeling rules as in case (ii).
\end{itemize}
See Fig.~\ref{fig:2D_windows}.

A  possible re-labeling of the points is $(1,1)\to 1$,  $(-1,1)\to 2$,  $(1,-1), (-1,-1)\to 3$.

\begin{figure}
\includegraphics[width=0.75\textwidth]{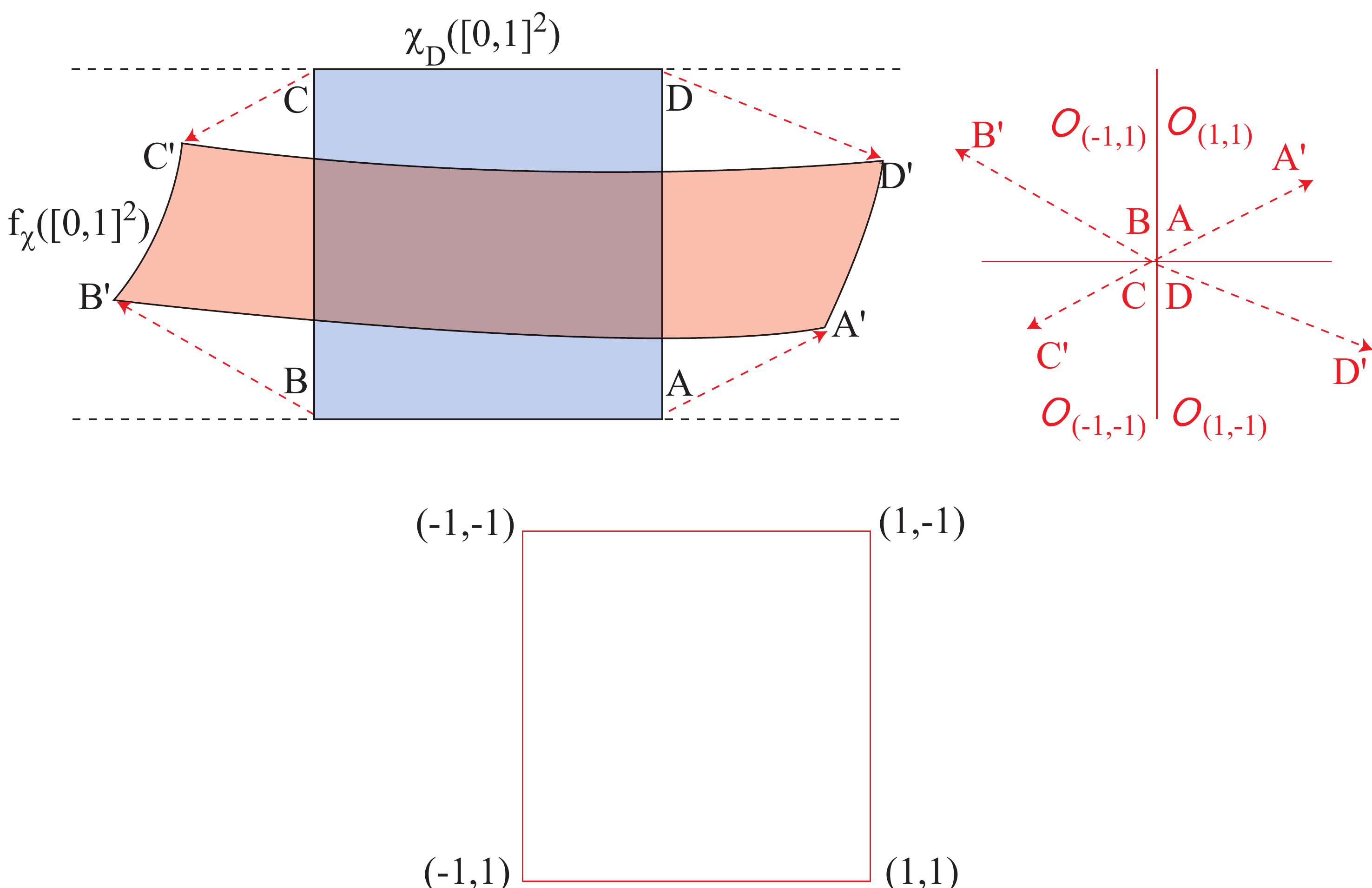}
\caption{2D correctly aligned windows and labeling of $[0,1]^2$}
\label{fig:2D_windows}
\end{figure}

\begin{prop}\label{prop:necessary_sufficient_2D}
If $D_1$ is correctly aligned with $D_2$ under $f$, then the   labeling of $\chi^{-1}_{D_1}(D_1)=[0,1]^2$  described above   { is \emph{non-degenerate}} and  has { \emph{non-zero index}}.

Conversely, in the case when $n_u=1$, $n_s=1$, if $D_1$, $D_2$ satisfy:
\begin{itemize}\item[(a)] $f_{\chi} ([0,1]^2 )\subset \R\times(0,1)$;
\item [(b)]  $f_{\chi}(\partial [0,1]\times [0,1])\cap [0,1]^2=\emptyset$;
 \end{itemize}
and if the labeling  as per \emph{Condition O} { is \emph{non-degenerate}}    and   of { \emph{non-zero index}}, then $D_1$ is correctly aligned with $D_2$ under $f$.
\end{prop}
\begin{proof}
For the direct statement, the
points  on the boundary of $[0,1]^2$ inherit the labeling explicitly described above, which is { non-degenerate}  and  has { non-zero index}.

For the converse statement, we proceed by contradiction. If condition (i.b) from Definition \ref{defn:win2d} is not satisfied, it means that the   two components of $\partial[0,1]\times [0,1]$ are mapped by $f_\chi$ on the same side of $[0,1]^2$ within the strip $\R\times(0,1)$, which leads to a labeling that fails to be { non-degenerate}.
\end{proof}

\begin{rem}We note that in  Proposition \ref{prop:necessary_sufficient_2D}, the converse statement does not include the cases that $n_u=0$, $n_s=2$, or $n_u=2$, $n_s=0$, since assuming    condition  (ii), or (iii) from Definition \ref{defn:win2d}, respectively, automatically yields both correct alignment and  { non-degenerate}, { non-zero index} labeling.\end{rem}

Assume that $f_\chi$ is a  bi-Lipschitz function with Lipschitz constant $L$, relative to the norm $\|\cdot\|_\infty$ on $\R^2$.
Consider a subdivision of $[0,1]^2$   into $N^2$ squares $\{C_i\}_{i=1,\ldots,N^2}$ of side $1/N$.
We assign a  labeling to all points of $[0,1]^2$ according to \emph{Condition O}.
Let $\delta>0$ be small, and $N>0$ be large so that $(L+1)/N<\delta$.  If $C_i$ is a completely labeled square, then any point $z\in C_i$ is a \emph{$\delta$-approximate fixed point}  of $f_\chi$.
Indeed, the complete labeling implies, via the Intermediate Value Theorem,  that there exist points $\hat z=(\hat x_1, \hat x_2)\in C_i$ and $\check z=(\check x_1, \check x_2)\in C_i$  such that $\hat x'_1 =\hat x_1$ and  $\check x'_2=\check x_2$, respectively.
Then, for each $z=(x_1,x_2)\in D_1$ we have
\begin{equation}\label{eqn:deltafixed}\begin{split}\|f_\chi(z)-z\|_\infty&=\max\{|x'_1-x_1|,|x'_2-x_2|\}\\&\leq \max\{|x'_1-\hat x'_1|+|\hat x'_1 -\hat x_1|+|\hat x_1-x_1|,
\\ &\quad\quad\quad\,\,\,|x'_2-\check x'_2|+|\check  x'_2 -\check  x_2|+|\check x_2-x_2|\}
\\\leq& \frac{(L+1)}{N}\\<&\delta.
\end{split}\end{equation}

The next statement is a fixed point theorem in the case of a window correctly aligned to itself. We will distinguish between the cases (i), (ii) of Definition \ref{defn:win2d}, and the case (iii), for which the corresponding statement is indicated in parentheses.
 In the former the labeling as per {Condition O} is done with respect to $f_\chi$, while in the latter the labeling is done with respect to $f^{-1}_\chi$.

\begin{prop}\label{prop:2dfixed}
Let $D$ be a window and $\phi:[0,1]^2=\chi_D^{-1}(D)\to\mathscr{Z}^2$ be a labeling associated to $f_\chi$ as per \emph{Condition O} (resp., associated to $f^{-1}_\chi$).

(i) If a window $D$ is correctly aligned with itself under $f$, then
$f$ has   a  fixed point in~$D$.

(ii) If $\{C_i\}$ is a subdivision of $[0,1]^2=\chi^{-1}_{D}(D)$ then there exists  a square $C_*$ in the decomposition { with $\ind_{C_*}(\phi)\neq 0$; if $C_*$ further satisfies the \emph{non-degeneracy condition} on its faces, then   $f$ has   a  fixed point in $\chi_D(C_*)$ (resp. $f$ has   a  fixed point in $\chi_D(f_{\chi}^{-1}(C_*))$).}

(ii) Assume that $\chi_D$ is Lipschitz with Lipschitz constant $K>1$, and that  $f_\chi$ is bi-Lipichitz with Lipschitz constant $L>1$. Then, given $\delta>0$ and a sufficiently fine subdivision of $[0,1]^2$ into squares $\{C_i\}_{i=1,\ldots,N^2}$ of side $1/N$, so that  $K(L+1)/N<\delta$,  for every completely labeled square $C_*$, each point $\tilde z\in \chi_D(C_*)$ is a    $\delta$-approximate fixed point  of $f$ (resp. each point $\tilde z\in \chi_D(f_{\chi}^{-1}(C_*))$ is a    $\delta$-approximate fixed point  of $f$).
\end{prop}

\begin{proof}
(i)  Let $\{C^N_i\}$  be a subdivision of $C$,  with $\textrm{diam}(C^N_i)\to 0$ as $N\to\infty$. For each $N$, by applying Corollary \ref{cor:sperner_cubical_1}, there exists {
a  square $C^N_{i^*_N}\subset C$ with $\ind_{C^N_{i^*_N}}(\phi)\neq 0$, hence completely labeled}; we choose and fix such a $C^N_{i^*_N}$.  As before, there exist some points $\hat z_{i^*_N}=((\hat x_1)_{i^*_N}, (\hat x_2)_{i^*_N})\in C^N_{i^*_N}$ and $\check z_{i^*_N}=((\check x_1)_{i^*_N}, (\check x_2)_{i^*_N})\in C^N_{i^*_N}$, such that $(\hat x_1)_{i^*_N}=(\hat x'_1)_{i^*_N}$ and $(\check x_2)_{i^*_N}=(\check x'_2)_{i^*_N}$.
By compactness, the sequences  $\hat z_{i^*_N}$ and $\check z_{i^*_N}$,
contain convergent subsequences  $\hat z_{i^*_{k_N}}$ and $\check z_{i^*_{k_N}}$, respectively.
Since the diameters of the corresponding completely
labeled squares tend to zero, these subsequences approach the same
limit $z=(x_1,x_2)$, for which we have   $x_1=x'_1$ and $x_2=x'_2$. Hence $z$ is a fixed point for $f_\chi$, so $\tilde{z}=\chi_D(z)$ is a fixed point for $f$.

For an alternative  proof (non-constructive) see \cite{GideaZ2004}.

(ii) Let $\{C_i\}$ be a subdivision as in the statement. { By Corollary \ref{cor:sperner_cubical_1}, there exists a square $C_*$ with $\ind_{C_*}(\phi)\neq 0$. If $C_*$ satisfies the non-degeneracy condition}, we   further subdivide $C_*$ into smaller squares  $\{C^N_i\}$ as above. The proof for (i) implies that there exists   a fixed point $z$ for $f_\chi$ in $C_*$, hence $p=\chi_D(z)$ is a fixed point for $f$  in $\chi_D(C_*)$.

(iii) Suppose that $C_{i^*}^N$ is a completely labeled square of the subdivision.

By \eqref{eqn:deltafixed}, for each $z\in C_{i^*}^N
$ we have $\|f_\chi(z)-z\|_\infty<(L+1)/N$. Hence, for each $\tilde z=\chi_D(z)\in \chi_D(C_{i^*}^N)$, we have
\[\|f(\tilde z)-\tilde z\|_\infty= \|\chi_D\circ f_\chi\circ \chi^{-1}_D(\chi_D(z))-\chi_D(z)\|_\infty\leq K\|f_\chi(z)-z\|_\infty<K(L+1)/N<\delta.
\]

In the case of labeling associated to $f^{-1}_\chi$, if $C_{i^*}^N$ is completely labeled, applying \eqref{eqn:deltafixed} to $f_\chi^{-1}$, and using that $f_\chi$ is bi-Lipschitz of Lipschitz constant $L$ yields $\|f^{-1}_\chi(z)-z\|_\infty<(L+1)/N$ for every $z\in C_{i^*}^N$. Thus $\|f_\chi(f^{-1}_\chi(z))- f^{-1}_\chi(z)\|_\infty< L(L+1)/N$, and so, for $\tilde z=\chi_D(f^{-1}_\chi(z))\in \chi_D(f^{-1}_\chi(C_{i^*}^N))$ we have $\|f(\tilde z)-\tilde z\|_\infty<KL(L+1)/N<\delta$.
\end{proof}


\subsection{Correctly aligned windows in higher-dimensions}

As before, let
\[\mathscr{O}_{\ell}=\{(x_1,\ldots,x_n)\in\mathbb{R}^n\,|\,\forall i,  x_i l_i\geq 0\},\]
where $\ell= (l_1,\ldots,l_n)\in\mathscr{Z}^n$, with $l_i=\pm 1$ for $i=1,\ldots,n$.
Let $g:\R^n\to \R^n$ be a homeomorphism.
For a point $z=(x_1,\ldots, x_n)$ in $\R^n$, letting $g(z)=z'=(x'_1,\ldots, x'_n)$, and $\Delta z=(x'_1-x_1,\ldots, x'_n-x_n)$, we assign a label $\ell\in\mathscr{Z}^n$ as follows:

\begin{description}
\item[Condition O]{$ $}

\begin{itemize}\item if $\Delta z\in\inte (\mathscr{O}_\ell)$ then we assign to $z$ the label  $\ell$;
\item if  $\Delta z\in \inte(\mathscr{O}_{\ell_1}\cap \ldots
\cap \mathscr{O}_{\ell_k})$ for some labels $\ell_1,\ldots,\ell_k$ that are mutually distinct, then we assign to $z$ either one of the labels $\ell_1,\ldots,\ell_k$.
\item if $\Delta z=0$ then we assign to $z$ either label $\ell$.
\end{itemize}
\end{description}

A  possible re-labeling of the points from labels in $\mathscr{Z}^n$ to labels in $\{1,\ldots,n+1\}$ can be done as in Example \ref{ex:relabeling}.

Let $D_1$, $D_2$ be $n$-dimensional windows. We will assume that the homeomorphism $f:\mathbb{R}^n\to\mathbb{R}^n$ satisfies a bi-Lipschitz condition with Lipschitz constant $L$.

\begin{defn}\label{defn:win-nd} We say that the window $D_1$ is correctly aligned with $D_2$ under $f$ if there exist  corresponding homeomorphisms $\chi_{D_1},\chi_{D_2}$ with the following properties:
\begin{itemize}
\item[(i)] Case $n_u\neq 0$, $n_s\neq 0$:
\begin{itemize}\item[(i.a)] $f_{\chi} ([0,1]^n )\subset \R^{n_u}\times(0,1)^{n_s}$;
\item [(i.b)]  $f_{\chi} (\partial [0,1]^{n_u} \times [0,1]^s  ) \subset \left(\R^{n_u}\times(0,1)^{n_s}\right)\setminus [0,1]^{n}$;
    \item [(i.c)]  There exists $x^*_s\in(0,1)^{n_s}$ such that  the map $L:[0,1]^{n_u}\to \R^{n_u}$ defined by $L(x_u)=\pi_u\circ f_\chi(x_u,x^*_s)$ satisfies $\deg(L_{\mid (0,1)^{n_u}})\neq 0$;
 \end{itemize}
\item[(ii)]
Case  $n_u=0$, $n_s=n$:
  $f_{\chi}( [0,1]^n )\subset
(0,1)^n$;
\item[(iii)] Case  $n_u=n$, $n_s=0$:
 $f^{-1}_{\chi} ([0,1]^n) \subset (0,1)^n$.
\end{itemize}
\end{defn}

Suppose that $D_1$ is correctly aligned with $D_2$ under $f$, and let $n_u$ be the unstable-like dimension.
We   denote by $\{\mathscr{O}_{\ell_u}\}$ the octants of $\mathbb{R}^{n_u}$, where $\ell_u\in\mathscr{Z}^{n_u}$.
In the  cases (i) and (ii) of Definition \ref{defn:win-nd}, we will label all points  of $C=\chi^{-1}_{D_1}(D_1)$    according to \emph{Condition O} applied to $f_\chi$, and in case (iii) of Definition \ref{defn:win-nd}  we will label all points  of $C=\chi^{-1}_{D_2}(D_2)$   with labels  according to \emph{Condition O} applied to $f_\chi$.
We describe the labeling below,  for each of the cases (i), (ii), (iii) of Definition \ref{defn:win-nd}.

\emph{Case (i).} First consider the case when $n_u>0$ and $n_s>0$.
 
Labeling $C=[0,1]^n$ as per \emph{Condition O},   as we did in the $2$-dimensional case  in Section \ref{sec:2D_windows}, does not necessarily yield a   \emph{non-degenerate} labeling. We will transform  $C=[0,1]^n$ into an $n$-dimensional polytope $\widetilde C$,   construct a subdivision  $\widetilde C=\bigcup_i C_i$ into smaller cubes,  in order to obtain a  \emph{non-degenerate}  labeling.  Moreover, we will construct  $\widetilde C$  so that the resulting labeling of its vertices in \emph{complete}.  (As we pointed out earlier, completeness is a necessary, but not sufficient condition for the index to be non-zero.)

The faces of the resulting polytope $\tilde C$ will consist of the faces of the $C_i$'s in the subdivision that are contained in $\partial([0,1]^n)$. We perform this construction below.

\emph{Estimates on $\Delta z$.}
Let \begin{equation*}\begin{split}\rho_s&=\min \{d(p,p')\,|\,p\in[0,1]^{n_u}\times\partial[0,1]^{n_s}, p'\in f_\chi([0,1]^{n}\times\partial[0,1]^{n_s})\}>0,\\
\rho_u&=\min \{d(p,p')\,|\,p\in\partial[0,1]^{n_u}\times[0,1]^{n_s}, p'\in f_\chi(\partial[0,1]^{n}\times[0,1]^{n_s})\}>0,\\
\rho&=\min\{\rho_s,\rho_u\}>0,
\end{split}\end{equation*}
where $d$ is the distance corresponding to $\|\cdot\|_\infty$. The fact that $\rho_u$, $\rho_s$, and hence $\rho$ are positive follows from Definition \ref{defn:win-nd} (i.a) and (i.b), respectively.  This fact implies that for each $z\in\partial[0,1]^n$ we have
\begin{equation}\label{eqn:delta_z}
\|\Delta z\|=\|z'-z\|_\infty=\|f_\chi(z)-z\|_\infty>\rho.
\end{equation}
In particular, for any point $z\in\partial[0,1]^n$ we have that $\Delta z$ is not contained within a ball of radius $\rho$ around the origin in $\R^{n}$.

\smallskip

\emph{Coarse cubical decomposition of $C$.} We divide $[0,1]^n$ into
$M^n$ identical cubes $\{C_i\}_{i=1,\ldots,M^n}$, of side $1/M$.
The quantity $M$ from above is required to satisfy the following condition:
\begin{description}
\item[Condition P]{$ $}

\begin{itemize}
\item[(P1)] For each $\ell_u\in\mathscr{Z}^{n_u}$, there exists a vertex $v$ of a cube $C_i$ lying on $\partial[0,1]^{n_u}\times[0,1]^{n_s}$, such that the  label of $z$ is  $\ell=(\ell_u,\ell_s)\in \mathscr{Z}^{n}$, for some $\ell_s\in \mathscr{Z}^{n_s}$.

\item[(P2)] $(L+1)/M<\rho/2$.
\end{itemize}
\end{description}

Condition (P1) implies that the cubical decomposition  $\{C_i\}_{i=1,\ldots,M^n}$  is fine enough so that for the vertices $z$ of the cubes with faces lying on $\partial[0,1]^{n_u}\times[0,1]^{n_s}$,
the vectors $\Delta z=(\Delta z_u,\Delta z_s)$ have the
$\Delta z_u$ component taking values in each  of the hyperoctants $\mathscr{O}_{\ell_u}$ of $\R^{n_u}$. The argument for this claim is below.

First we note Definition \ref{defn:win-nd}-(i.a)  implies that the  corresponding $\pi_s(\Delta z)$ take values in each  of the sectors $\mathscr{O}_{\ell_s}$ of $\R^{n_s}$.
Definition \eqref{defn:win-nd}-(i.b) and -(i.c) imply that,   for some $x^*_s\in(0,1)^{n_s}$  the projection $\pi_{u}$ onto $[0,1]^{n_u}$  of the image of $[0,1]^{n_u}\times\{x ^*_s\}$ under $f_\chi$ contains the rectangle $[0,1]^{n_u}$ inside its interior, and that the boundary of $\pi_u(f_\chi([0,1]^{n_u}\times\{x^*_s\}))$ wraps around the boundary of $[0,1]^{n_u}$, in the sense that for $z\in\partial\left[ \pi_u(f_\chi([0,1]^{n_u}\times\{x^*_s\}))\right]$, the corresponding $\pi_u(\Delta  z)$ visits all sectors $\mathscr{O}_{\ell_u}$ of $\R^{n_u}$.

It follows that $\Delta z=(\Delta  z_u,\Delta  z_s)$ take values in a complete set of hyperoctants $\mathscr{O}_{\ell}$ of $\R^{n}$. (This does not mean that $\Delta z$ takes values in  all hyperoctants $\mathscr{O}_{\ell}$ of $\R^{n}$.)
Thus, the corresponding labeling of the vertices of the $C_i$'s is \emph{complete}.

We append the vertices and faces $S_i$ of the $C_i$'s that lie on $\partial[0,1]^n$ to $C$, thus transforming $C$ into a polytope $\widetilde C$ (like a Rubik's cube, see Fig.~\ref{fig:rubik}).
\begin{figure}
\includegraphics[width=0.25\textwidth]{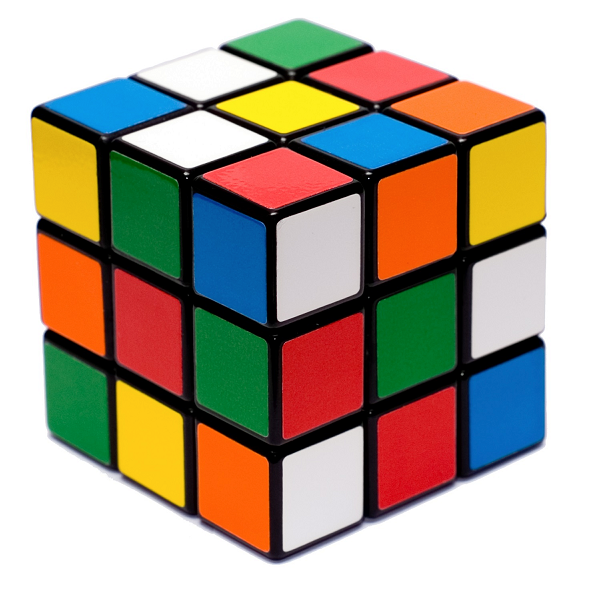}
\caption{Coarse decomposition of $C$ and transformation into a polytope $\widetilde C$.}
\label{fig:rubik}
\end{figure}

Now we discuss  Condition (P2). Note first that for $z_1,z_2\in \partial[0,1]^{n}$, we have \begin{equation}\begin{split}\|\Delta z_1- \Delta z_2\|_\infty&=\|(z'_1-z_1)-(z'_2-z_2)\|_\infty\leq \|z'_1-z'_2\|_\infty+\|z_1-z_2\|_\infty\\&\leq (L+1)\|z_1-z_2\|_\infty.
\end{split}\end{equation}
This implies that, the image of any cube $C_i$ under the map $z\mapsto \Delta z$ has diameter less than $\rho/2$. Hence,  the image under $z\mapsto \Delta z$ of every face $S_i$ of a cube $C_i$ that lies on  $\partial[0,1]^{n}$, is disjoint from a $\rho$-ball around the origin. Hence no such a face $S_i$ can carry a complete set of labels. That is, the labeling is \emph{non-degenerate}.

When the windows are correctly aligned, as assumed above, it also follows that the index of the labeling,  is non-zero.
Condition (i.a) of correct alignment implies that the index relative  to the labels in $\mathscr{Z}^{n_s}$ is non-zero. Also, conditions (i.b) and (i.c), together with (P1), imply that the
the index relative to the labels in $\mathscr{Z}^{n_u}$ is non-zero. Proposition \ref{prop:Bekker}-(iii), saying that the index of a labeling equals the Brouwer degree of a realization, and  the product property of the Brouwer degree, imply that the overall labeling is non-degenerate. 

\emph{Case (ii).} Consider the case when $n_u=0$.  We label all points of $C=\chi^{-1}_{D_1}(D_1)$ according to the quadrant $\mathscr{O}_\ell$ where $f_{\chi}(z)-z$ lands, as per \emph{Condition O}. The resulting labeling is \emph{non-degenerate} and of \emph{non-zero index}.

\emph{Case (iii).} Consider the case when $n_s=0$. We label all points of $C=\chi^{-1}_{D_2}(D_2)$ according to the quadrant $\mathscr{O}_\ell$ where $f^{-1}_{\chi}(z)-z$ lands, as per \emph{Condition O}. The resulting labeling is \emph{non-degenerate} and of \emph{non-zero index}.

\begin{ex} In Fig.~\ref{fig:3D_windows} we illustrate a 3D  window that is correctly aligned to itself under some map; the points $A,B,\ldots$, are mapped to the points $A',B',\ldots$, respectively. The unstable-like dimension is $n_u=2$ and the stable-like dimension is $n_s=1$.  If we label the vertices $A,B,\ldots$ according to \emph{Condition O}, we obtain
$A\to \ell_1:=(-1,1,1)$, $B\to \ell_2:=(-1,1,1)$, $C\to \ell_3:=(-1,-1,1)$, $D\to \ell_4:=(-1,-1,1)$,
$E\to \ell_5:=(-1,1,-1)$, $F\to \ell_6:=(-1,1,-1)$, $G\to \ell_7:=(-1,-1,-1)$, $H\to \ell_8:=(-1,-1,-1)$.
We notice that the corresponding set of labels is not complete; indeed, for the corresponding octants we have that
$\mathscr{O}_{\ell_1}\cap\ldots\cap \mathscr{O}_{\ell_8}=\{x\in \R^n\,|\, x_1\leq 0, x_2=0,x_3=0\}$. { In particular, the index is zero.}
The labels $\ell_u$ corresponding to the unstable directions take only the values $(-1,1)$ and $(-1,-1)$, so \emph{Condition P} is not satisfied. However, by taking a coarse cubical decomposition of $[0,1]^n$ satisfying \emph{Condition P}, as illustrated in Fig.~\ref{fig:3D_windows}, we can obtain that the labeling is \emph{non-degenerate} and { has non-zero index}.

\begin{figure}
\includegraphics[width=0.75\textwidth]{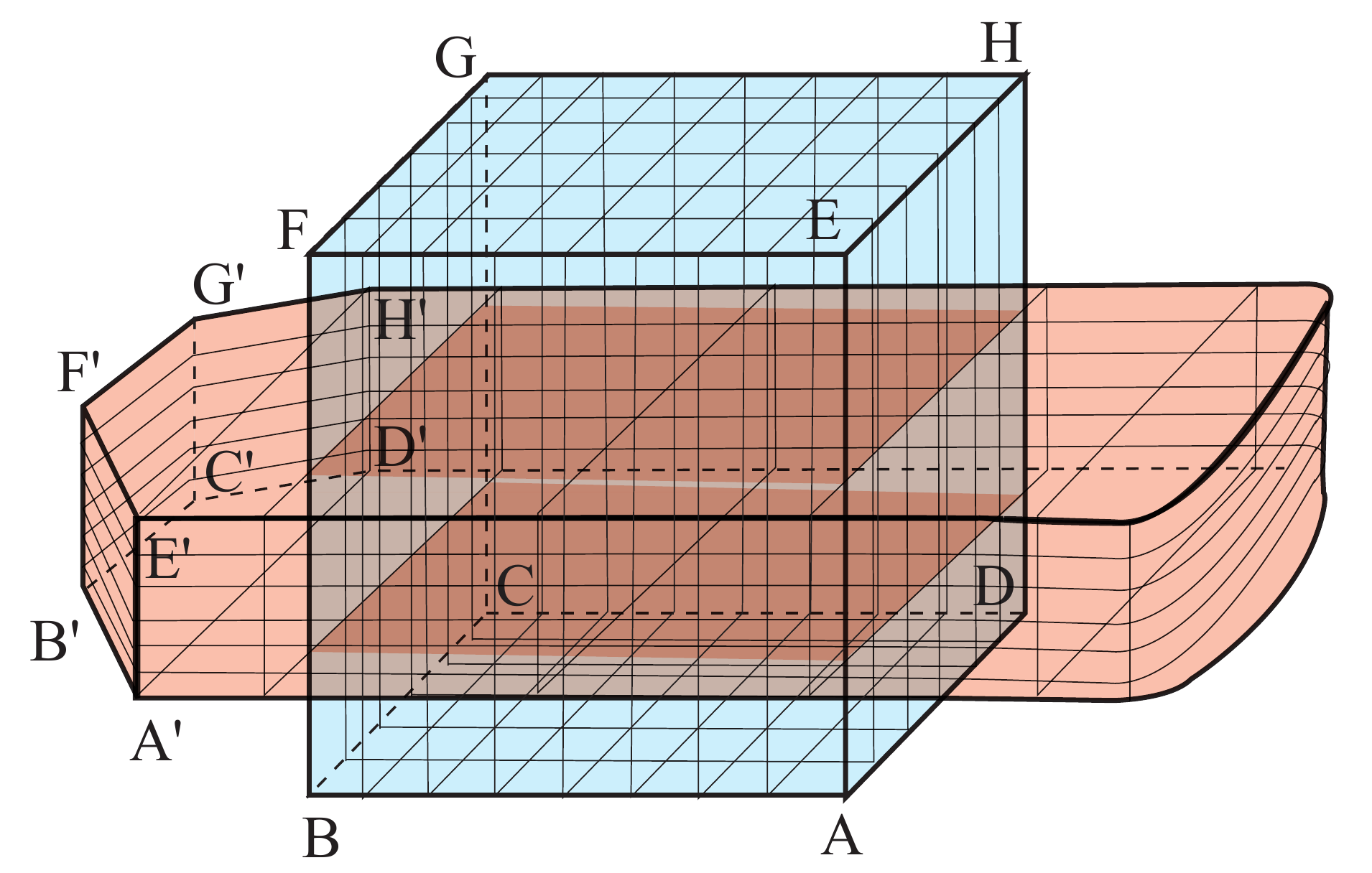}
\caption{3D correctly aligned windows, and cubical decomposition satisfying \emph{Condition P}.}
\label{fig:3D_windows}
\end{figure}
\end{ex}

To summarize, if  $D_1$ is correctly aligned with $D_2$ under $f$,   in case (i) and (ii) we start with $C=\chi_{D_1}^{-1}(D_1)$ and assign a labeling associated  to $\Delta z=f_\chi(z)-z$, and in case (iii) we start with $C=\chi_{D_2}^{-1}(D_2)$ and assign a labeling associated  to $\Delta z=f^{-1}_\chi(z)-z$.
In each case, we perform a cubical decomposition of $C$ into smaller cubes $\{C_i\}_i$, and transform $C$ into a polytope $\widetilde C$ with a cubical decomposition $\{C_i\}_i$ such that the labeling is \emph{non-degenerate} and  has \emph{non-zero index}.

\begin{prop}\label{prop:necessary_sufficient_alignment_nD} Assume  $D_1$ is correctly aligned with $D_2$ under $f$, and  $\widetilde C=\{C_i\}$ is a subdivision of $[0,1]^n=\chi^{-1}_{D}(D)$ satisfying \emph{Condition P}. Then the labeling described above, in each of the cases (i), (ii), (iii) of Definition \ref{defn:win-nd},  satisfies
\[\ind_{\widetilde C}(\phi)\neq 0.\]

Conversely, in the case when $n_u>0$, $n_s>0$, if $D_1$, $D_2$ satisfy:
\begin{itemize}\item[(a)]   $f_{\chi} ([0,1]^n )\subset \R^{n_u}\times(0,1)^{n_s}$;
\item [(b)]  $f_{\chi} (\partial [0,1]^{n_u} \times [0,1]^{n_s}  ) \subset \left(\R^{n_u}\times(0,1)^{n_s}\right)\setminus [0,1]^{n}$;
 \end{itemize}
and the corresponding decomposition $\widetilde C=\bigcup _i C_i$ satisfies $\ind_{\widetilde C}(\phi)\neq 0$, then $D_1$ is correctly aligned with $D_2$ under $f$.
\end{prop}
\begin{proof}
The direct statement was shown above.

For the converse statement, pick any $x^*_s\in(0,1)^{n_s}$ and consider the labeling of the $n_s$-dimensional polytope  $[0,1]^{n_u}\times\{x^*_s\}$ as per \emph{Condition O}. As before, it follows that the index of the labeling, relative to the labels in $\mathscr{Z}^{n_u}$, is non-zero. By Proposition \ref{prop:Bekker}, this index equals to the Brouwer degree of a realization $\Phi:[0,1]^{n_u}\times\{x^*_s\}\to T^{n_u}$, where $T^{n_u}$ is the $(n_u)$-dimensional simplex. This degree is non-zero and equals, up to a sign,   the degree of $L(\cdot)=\pi_u\circ f_\chi(\cdot,x^*_s)$, which concludes the proof.

\end{proof}

In the statement below, we distinguish between the cases (i), (ii) of Definition \ref{defn:win-nd}, and the case
(iii), for which the corresponding statement is indicated in parentheses.

\begin{prop}\label{prop:ndfixed}
Let $D$ be a window and $\phi:\chi_D^{-1}(D)\to\mathscr{Z}^n$ be a labeling associated to $f_\chi$ as per \emph{Condition O} (resp., associated to $f^{-1}_\chi$). Assume that  $\widetilde C=\{C_i\}$ is a subdivision of $[0,1]^n=\chi^{-1}_{D}(D)$ satisfying \emph{Condition P} and \emph{Condition O}.

(i) If $D$ is correctly aligned with itself under $f$, then $f$
has a fixed point in $ D $.

(ii) Let  $\{C'_j\}_{j=1,\ldots,N^n}$ be a    fine subdivision of $[0,1]^n$ into cubes of side $1/N$, where $N$ is a multiple of  $M$ (so that the family of cubes $C'_j$ subdivide each of the cubes $C_i$). Then there exists  a  cube $C'_*$ with $\ind_{C'_*}(\phi)\neq 0$ in the decomposition; if $C'_*$ further satisfies the \emph{non-degeneracy condition}, then   $f$ has   a  fixed point $p$ in $\chi_D(C'_*)$ (resp., $f$ has a fixed point in $\chi_{D}(f_\chi^{-1}(C'_*))$).

(iii) Assume that $\chi_D$ is Lipschitz with Lipschitz constant $K>1$, and that  $f_\chi$ is bi-Lipichitz with Lipschitz constant $L>1$. Then, given $\delta>0$ and a   subdivision   $\{C'_j\}_{j=1,\ldots,N^n}$ of $[0,1]^n$  into cubes  of side $1/N$ as above, so that  $K(L+1)/N<\delta$, then for every completely labeled cube $C'_j$, each point $\tilde z\in \chi_D(C'_j)$ is a    $\delta$-approximate fixed point  of $f$ (resp., each point $\tilde z\in \chi_{D}(f_\chi^{-1}(C'_j))$).
\end{prop}

\begin{proof} The proof is similar to the proof of Proposition \ref{prop:2dfixed}, and the details are left to the reader.
\end{proof}

\subsection{Detection of periodic points and symbolic dynamics}
Assume that $p_1$ is a periodic point of period $k$ for $f$; the orbit of $p_1$ is $\{p_1,\ldots,p_k\}$, with $f(p_k)=p_1$.
Let $F:(\R^n)^k\to (\R^n)^k$ be given by
\begin{equation}
\label{eqn:F}
F(z_1,\ldots, z_k)=(f(z_k), f(z_1),\ldots,f(z_{k-1})).
\end{equation}
Note that   $\{p_1,\ldots,p_k\}$ is a period-$k$ orbit if and only if  $F(p_1,\ldots,p_k)=(p_1,\ldots, p_k)$, that is, $(p_1,\ldots,p_k)\in (\R^n)^k$ is a fixed point for $F$.

Now consider a finite sequence of windows $D_1,\ldots, D_k$ in $\R^n$. We are interested in periodic orbits $\{p_1,\ldots,p_k\}$ with  $p_j\in D_{j}$, $j=1,\ldots, k$. Assume that for $j=1,\ldots, k-1$, $D_j$ is  correctly aligned with $D_{j+1}$ under $f$, and also $D_k$ is  correctly aligned with $D_{1}$ under $f$.
{ Here we only consider correct alignment as in Definition \ref{defn:win-nd}-(i).}
See Fig.~\ref{fig:windows}.
Denote by $\chi_{D_j}$, the equivalence class of homeomorphisms corresponding to $D_j$, for $j=1,\ldots,k$. Let $\chi_D:(\R^n)^k\to (\R^n)^k$ be given by $\chi_D(z_1,\ldots, z_k)=(\chi_{D_1}(z_1),\ldots,\chi_{D_k}(z_{k}))$.

\begin{lem}\label{lem:product_aligned}
(i) Let $D=\chi_D\left(\Pi_{i=1}^{k} [0,1]^n\right)\subseteq  (\R^n)^k$, and
\begin{equation}\begin{split}
D^{-}&= \chi_D\left(\Pi_{j=1}^{k}\partial[0,1]^{n_u}\times[0,1]^{n_s}\right),\\
D^{+}&= \chi_D\left(\Pi_{j=1}^{k} [0,1]^{n_u}\times\partial[0,1]^{n_s}\right).
\end{split}\end{equation}
Then $D$ is an $(nk)$-dimensional window, with $(n_uk)$-unstable-like, and $(n_sk)$-stable like dimensions.

(ii) If for $j=1,\ldots, k-1$, $D_j$ is  correctly aligned with $D_{j+1}$ under $f$, and   $D_k$ is  correctly aligned with $D_{1}$ under $f$, then $D$ is correctly aligned with $D$ under $F$.
\end{lem}
\begin{proof} (i) Follows from elementary set theory. (ii) Follows from  elementary set theory and from the product property of the Brouwer degree. See \cite{GideaZ2004}.
\end{proof}

\begin{figure}
\includegraphics[width=1.0\textwidth]{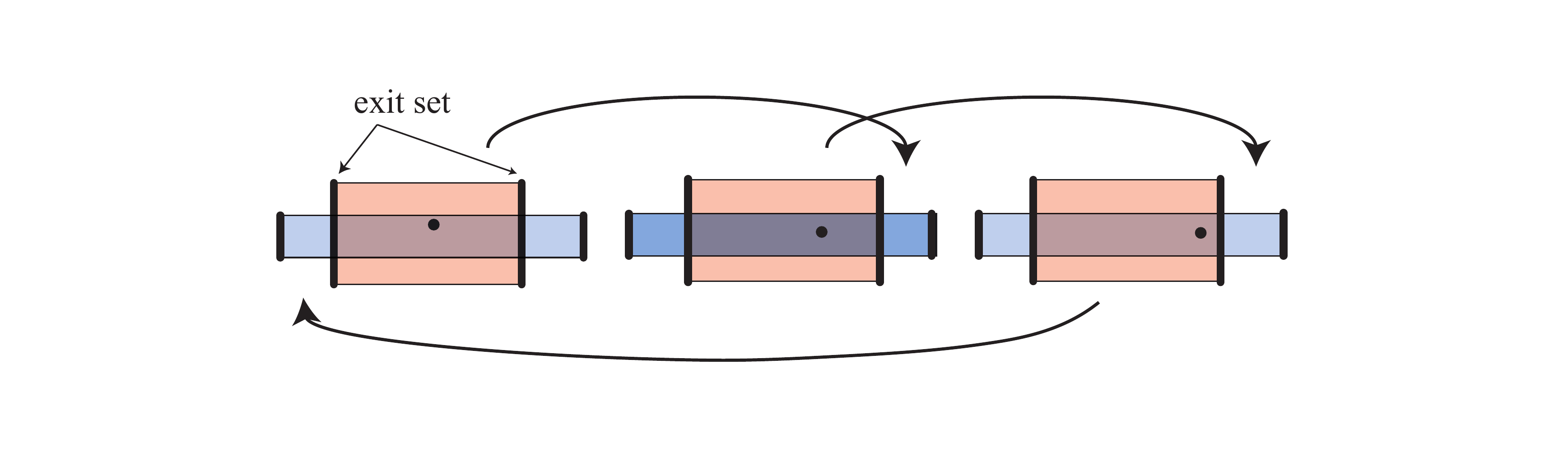}
\caption{Periodic sequence of  correctly aligned windows.}
\label{fig:windows}
\end{figure}

We associate to each rectangle $D_j$, $j=1,\ldots,k$,  an  $n$-dimensional polytope $P_j$, obtained by dividing each underlying cube $\chi^{-1}_{D_j}(D_j)=[0,1]^n$ into $(N_j)^n$ cubes of side $1/N_j$, where $N_j$ is chosen large enough so that \emph{Condition P} is satisfied.

The cubical decomposition of each window $D_j$ determines a { \emph{coarse rectangular decomposition} of $\chi^{-1}_D(D)=([0,1]^{n})^{k}$ into multi-dimensional rectangles} of the form
\begin{equation}\label{eqn:cubic_coarse}\begin{split}C_{\alpha}=&(C_1)_{\alpha_1}\times (C_2)_{\alpha_2}\times \cdots \times (C_k)_{\alpha_k},\\&\textrm { where } {\alpha}=(\alpha_1,\ldots, \alpha_k)\in \mathscr{A}:=\{1,\ldots,(M_1)^n\}\times\cdots \times\{1,\ldots,(M_k)^n\}.\end{split}  \end{equation}

Further, we divide each $D_j$ into small cubes $\{(C_j)_\beta\}_{\beta=1,\ldots,(N_j)^n}$, of side $1/N_j$, where $N_j$ is a multiple of $M_j$, obtaining a { \emph{fine rectangular decomposition}}. For each vertex $z_j$ of a cube $(C_j)_\beta$ in the cubical decomposition of $D_j$, we assign a label $\ell_j=(\pm 1, \ldots, \pm 1)\in\mathscr{Z}^n$, based on the sector of $\mathscr{O}_\ell\subseteq\R^n$ where the displacement vector $\Delta z_j=f_{\chi_{D_j,D_{j+1}}}(z_j)-z_j$ lands.

Relative to the product window $D$, this can also be regarded as an $(nk)$-dimensional polytope.
The cubical decomposition of each window $D_j$ determines a { rectangular} decomposition of $D$ of the form
\begin{equation}\label{eqn:cubic_fine}\begin{split}C_{\beta}=&(C_1)_{\beta_1}\times (C_2)_{\beta_2}\times \cdots \times (C_k)_{\beta_k},\\&\textrm { where } {\beta}=(\beta_1,\ldots, \beta_k)\in\mathscr{B}:=\{1,\ldots,(N_1)^n\}\times\cdots \times\{1,\ldots,(N_k)^n\}.\end{split}  \end{equation}
Note that for each $\alpha \in \mathscr{A}$, $(C_\alpha \cap C_\beta)_{\beta\in\mathscr{B}}$ forms a { rectangular} decomposition of  $C_\alpha$.

Each vertex $z=(z_1,\ldots,z_k)$ of a cube $C_{\beta}$ is assigned a label $\ell=(\ell_1,\ldots,\ell_k)\in(\mathscr{Z}^{n})^{k}$ whose component $\ell_j\in \mathscr{Z}^{n}$ is the label corresponding to the vertex $z_j$, $k=1,\ldots, k$, according to \emph{Condition O}.

\begin{prop}\label{prop:periodic-orbit-higher-dimensions}
(i) Given a sequence of windows $D_1,\ldots, D_k$ as above, with $D_j$  correctly aligned with $D_{j+1}$ under $f$, for $j=1,\ldots, k$, and   $D_k$   correctly aligned with $D_{1}$ under $f$. Then
there exists a periodic orbit $p_1,\ldots,p_k$ of $f$ of period $k$ with $p_j\in \textrm{int}(D_j)$, for $j=1,\ldots,k$.

(ii) If $\{(C_j)_{\alpha}\}$ is a coarse subdivision of $[0,1]^n=\chi^{-1}_{D_j}(D_j)$ then for each  $j$ there exists  $(C_j)_{\alpha^*_j}$  with $\ind_{(C_j)_{\alpha^*_j}}(\phi)\neq 0$; if each $(C_j)_{\alpha^*_j}$ further satisfies the \emph{non-degeneracy condition} on its faces, then   $f$ has  a periodic orbit $p_1,\ldots,p_k$ with $p_j\in \chi_{D_j}((C_j)_{\alpha^*_j})$,  for $j=1,\ldots,k$.

(iii) Assume that $\chi_{D_j}$ is Lipschitz with Lipschitz constant $K_j>1$, and that  $f_{\chi_{D_j,D_{j+1}}}$ is bi-Lipschitz with Lipschitz constant $L_j>1$. Let $\delta>0$ and consider a sufficiently fine subdivision of each $\chi_{D_j}^{-1}(D_j)=[0,1]^n$ into cubes $\{(C_j)_{\beta_j}\}_{\beta_j=1,\ldots,N_j^n}$ as above, so that  $\max_j\left\{\frac{K_{j+1}(L_j+1)}{N}\right\}<\delta$. Then for every sequence of cubes $(C_j)_{\beta^*_j}\subseteq \chi^{-1}_{D_j}(D_j)$ that are completely labeled,  every sequence of points $\tilde z_j=\chi_{D_j}(z_j)\in \chi_{D_j}( (C_j)_{\beta^*_j})$, $j=1,\ldots,k$, is a $\delta$-approximate periodic orbit of period $k$.
\end{prop}

\begin{proof}
(i) By Lemma \ref{lem:product_aligned}, the correct alignment of $D_j$ with $D_{j+1}$ under $f$ implies that $D$ is correctly aligned with itself under $F$. The cubical decompositions  $\{(C_j)_\alpha\}$  of $[0,1]^n=\chi^{-1}_{D_j}(D_j)$, $j=1,\ldots,k$  determine a coarse decomposition $C_\alpha$ of $([0,1]^{n})^{k}$ as in \eqref{eqn:cubic_coarse}. By Theorem \ref{thm:sperner_cubical}, there exists a  cube $C_{\alpha_*}=(C_1)_{\alpha^*_1}\times(C_2)_{\alpha^*_2}\times\ldots\times (C_k)_{\alpha^k_*}$ in this decomposition with $\ind_{C_{\alpha_*}}(\phi)\neq 0$. The labeling   of the vertices of $C_{\alpha^*}$ with respect to the map $F_{\chi_D}$, induce a labeling  of each  cube  $(C_j)_{\alpha^*_j}$ with respect to the corresponding map $f_{\chi_{D_j},\chi_{D_{j+1}}}$ such that $\ind_{(C_j)_{\alpha^*_j}}(\phi)\neq 0$.

Take now a sequence of fine cubical decomposition $(C_j)_{\beta_j^N}$ of $\chi^{-1}_{D_j}(D_j)$,  as in \eqref{eqn:cubic_fine}, with $\textrm{diam}((C_j)_{\beta_j^N})\to 0$ as $N\to\infty$. By the above argument, within each subdivision one can find a  cube $(C_j)_{\beta_{j}^{*N}}$ with $\ind_{(C_j)_{\beta_{j}^{*N}}}(\phi)\neq 0$. Since such labeling is also complete, it implies  that for each $i\in\{1,\ldots,n\}$ there exists an $n$-tuple of points  $ z^{i,N}_j\in (C_j)_{\beta_{j}^{*N}}$, such that for each $i$ we have that $\pi_i\left(f_{\chi}( z^{i,N}_j)- z^{i,N}_{j+1}\right)=0$. By successively extracting convergent subsequence in each $i$-coordinate, for each $j$ we obtain $n$ subsequences in $(C_j)_{\alpha^*_j}$ of the $  z^{i,N}_j$'s  that are convergent to the same limit $  z_j\in (C_j)_{\alpha^*_j}$, and such that $\pi_i\left(f_{\chi}( z_j)- z_{j+1}\right)=0$ for all $i=1,\ldots,n$, that is, $f_\chi( z_j)= z_{j+1}$. Hence
$p_j=\chi_{D_j}( z_j)$, $j=1,\ldots,k$, is a periodic sequence of period $k$ for $f$.

(ii) Using the non-degeneracy condition on the labeling and that  $\ind_{(C_j)_{\alpha^*_j}}(\phi)\neq 0$ for $j=1,\ldots,k$,  we apply the previous argument to the collection of cubes $(C_j)_{\alpha^*_j}$, obtaining a periodic orbit $p_j\in \chi_{D_j}((C_j)_{\alpha^*_j})$ for $f$, where $j=1,\ldots,k$.

(iii) Consider the points $z^{i,N}_j\in (C_j)_{\beta^*_j}$ as before, with $N$ large enough as in the statement.  Let $z_j$ be an arbitrary point in $(C_j)_{\beta^*_j}$, and $\hat z_j=\chi_{D_j}(z_j)$, for $j=1,\ldots,k$.
We have \begin{equation*}\begin{split}\|f(\hat z_j)-\hat z_{j+1}\|_\infty &=\|\chi_{D_{j+1}}\circ f_{\chi_{D_{j}}, \chi_{D_{j+1}}}\circ \chi^{-1}_{D_{j}}(\chi_{D_{j}}(z_{j}))-\chi_{D_{j+1}} (z_{j+1})\|_\infty\\&\leq K_{j+1}\|f_{\chi_{D_{j}}, \chi_{D_{j+1}}}(z_j)-z_{j+1}\|_\infty\\
&= K_{j+1}\max_{i=1,\ldots, n}|\pi_i(f_{\chi_{D_{j}}, \chi_{D_{j+1}}}(z_j))-\pi_i(z_{j+1})|\\
&\leq K_{j+1}\max_{i=1,\ldots,n}\left (
\left|\pi_i(f_{\chi_{D_{j}}, \chi_{D_{j+1}}}(z_j))-\pi_i(f_{\chi_{D_{j}}, \chi_{D_{j+1}}}(z^{i,N}_j))\right|\right .
\\&{}\qquad\qquad +
\left . \left|\pi_i(f_{\chi_{D_{j}}, \chi_{D_{j+1}}}(z^{i,N}_j) -\pi_i(z^{i,N}_{j+1})\right|\right .\\&\qquad\qquad +
\left . \left|\pi_i( z^{i,N}_{j+1})-\pi_i(z_{j+1})\right|\right )\\
&\leq\frac{K_{j+1}(L_j+1)}{N}\\&<\delta.
\end{split}\end{equation*}
Thus, $\hat z_1,\ldots, \hat z_k$   is a $\delta$-approximate periodic orbit of period $k$ for $f$.
\end{proof}

\begin{prop}\label{prop:symbolic-dynamics-higher-dimensions}
(i) Assume that $D_1,\ldots, D_k$ is a sequence of windows as above. Let $\Gamma=(\gamma_{ij})_{i,j=1,\ldots,k}$ be a transition matrix, where $\gamma_{ij}=0$ or $1$; assume that for any $i,j$ with $\gamma_{ij}=1$, $D_i$ is correctly aligned with $D_j$ under $f$. Consider the topological Markov chain associated to the transition matrix $\Gamma$ defined by
\[\Omega_\Gamma=\{\omega:=(\omega_t)_{t\in\mathbb{Z}}\,|\,\omega_t\in\{1,\ldots,k\}  \textrm { and } \gamma_{\omega_t\omega_{t+1}}=1\textrm{ for all } t\},\]
and the shift map
$\sigma:\Omega_\Gamma\to\Omega_\Gamma$ given by  $(\sigma(\omega))_t=\omega_{t+1}$,  $t\in\mathbb{Z}$.
Then, for every sequence $\omega \in\Omega_\Gamma$, there exists an orbit $(p_t)_{t\in\mathbb{Z}}$ of $f$, with  $p_t:=f^t(p_0)\in \textrm{int}(D_{\omega_t})$, for~all~$t$.

(ii) Assume that $\chi_{D_j}$ is Lipschitz with Lipschitz constant $K>1$, and that  $f_{\chi_{D_{j},D_{l}}}$ is bi-Lipschitz with Lipschitz constant $L>1$, for all $j,l\in\{1,\ldots,k\}$. Let $\delta>0$, $T\in\mathbb{Z}^+$, and $\omega\in \Omega_\Gamma$. Consider a sufficiently fine subdivision of each $\chi_{D_j}^{-1}(D_j)=[0,1]^n$ into cubes $\{(C_j)_{\beta_j}\}_{\beta_j=1,\ldots,N^n}$ as above, so that  $\max_j\left\{\frac{K (L+1)}{N}\right\}<\delta$, $j=1,\ldots,k$.
Then for every sequence of cubes $(C_{\omega_t})_{\beta^*_t}\subseteq \chi^{-1}_{D_{\omega_t}}(D_{\omega_t})$ that are completely labeled,  every sequence of points $\tilde z_t=\chi_{D_{\omega_t}}(z_{\omega_t})\in \chi_{D_{\omega_t}}( (C_{\omega_t})_{\beta^*_t})$, $t=1,\ldots,T$, is a $\delta$-approximate orbit  of length $T$, in the following sense
\[d(f(\tilde z_t), \tilde z_{t+1})<\delta, \textrm { for } t=1,\ldots,T.\]
\end{prop}
\begin{proof} (i) It is enough  to prove that for each $\omega\in\Omega_\Gamma$, for the infinite of  windows $D_{\omega_t}$, $t\in\mathbb{Z}$, there exists a point $p_0$ in $D_{\omega_0}$ such that $f^t(p_0)\in D_{\omega_t}$.
This follows from the following:

\emph{Claim 1.  If  $\{D_t\}_{t=1,\ldots, k}$, is a sequence of windows such that for every $t=1,\ldots, k-1$, $D_t$ is correctly aligned with $D_{t+1}$ under $f$, then there exists an orbit $(p_t)_{t=1,\ldots,k}$ such that $p_{t+1}=f(p_t)$, and $p_t\in D_t$ for all $t$.}

\emph{Proof of Claim 1.} We can always define a continuous map $\widehat f$ such that $D_n$ is correctly aligned with $D_0$ under $\widehat f$. Then, similarly to \eqref{eqn:F} we define the map
\begin{equation}
\label{eqn:F}
\widehat F(z_1,\ldots, z_k)=(\widehat f(z_k), f(z_1),\ldots,f(z_{k-1})).
\end{equation}
Denoting $\chi_D(z_1,\ldots, z_k)=(\chi_{D_1}(z_1),\ldots,\chi_{D_k}(z_{k}))$, as in Lemma \ref{lem:product_aligned} and Proposition \ref{prop:periodic-orbit-higher-dimensions}, we obtain that $\chi_D(\prod_{t=1}^{k} [0,1]^n)$ is correctly aligned to itself under $\widehat F$, hence there is a fixed point for $\widehat F$. This yields an orbit of $f$ as in the claim; the fact that $p_1=\widehat{f}(p_k)$ is irrelevant for the dynamics.

\emph{Claim 2.  If  $D_t$, $t\in\mathbb{Z}$, is a sequence of windows such that for every $t$, $D_t$ is correctly aligned with $D_{t+1}$ under $f$, then there exists an orbit $(p_t)_{t\in\mathbb{Z}}$ such that $p_{t+1}=f(p_t)$, and $p_t\in D_t$ for all $t$.}

\emph{Proof of Claim 2.} By \emph{Claim 1} for each  finite sequence of windows \[D_{-N}, \ldots,D_0,\ldots, D_{N},\] there is a point $p_0^N\in D_0$ such that   $f^t(p^N_0)\in D_t$ for all $t\in\{-N,\ldots, N\}$. Taking a convergent subsequence  $p^{k_N}_0$ of $p^N_0$ with  $p^{k_N}_0\to p_0$ as $N\to\infty$, we obtain that the orbit of $p_0$ is as claimed.

(ii) The proof follows in the same way as for Proposition \ref{prop:periodic-orbit-higher-dimensions}-(ii).
\end{proof}

For a related statement to Proposition \ref{prop:symbolic-dynamics-higher-dimensions} see \cite{gidea1999conley}.

\section{Application}\label{sec:application}
In this section we illustrate the methodology developed in this paper on a simple example. Namely, we consider the H\'enon Map, defined as $f(x,y)=(a-x^2+by,x)$ for   $a=1.25$  and $b=0.3$.  We will use the Sperner lemma-based approach to show the existence of  a period-$7$ orbit.

\begin{figure}
\includegraphics[width=0.75\textwidth]{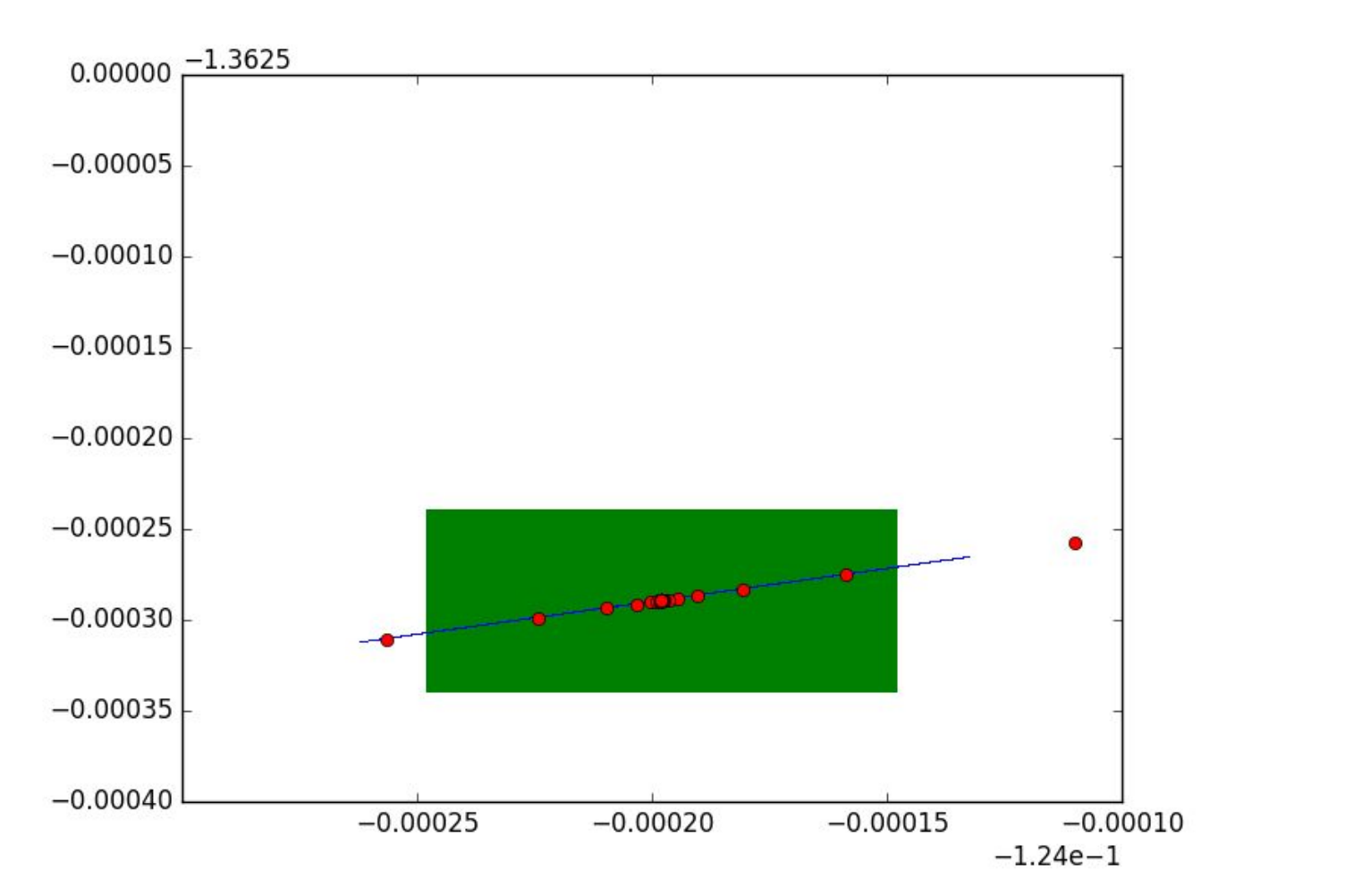}
\caption{The window $D$ (green) and its image $f^7(D)$ (blue) under the seventh iterate of the map.}
\label{fig:window_seven_iterate}
\end{figure}
\begin{figure}
\includegraphics[width=0.75\textwidth]{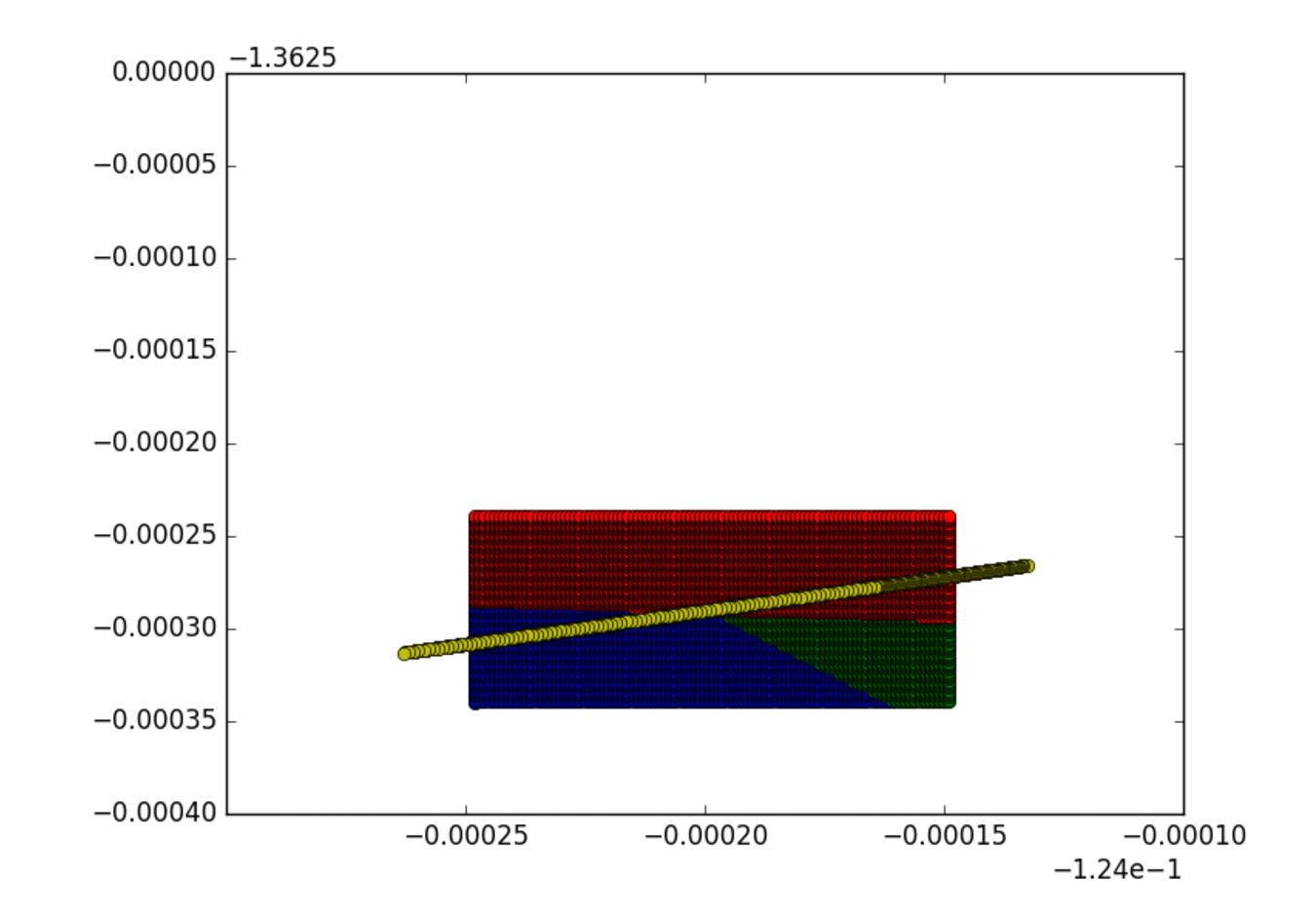}
\caption{Sperner labeling of $D$; the label $1$ is shown in blue, $2$ in green, and $3$ in red.}
\label{fig:Sperner_labeling}
\end{figure}

\begin{figure}[h]
\includegraphics[width=0.75\textwidth]{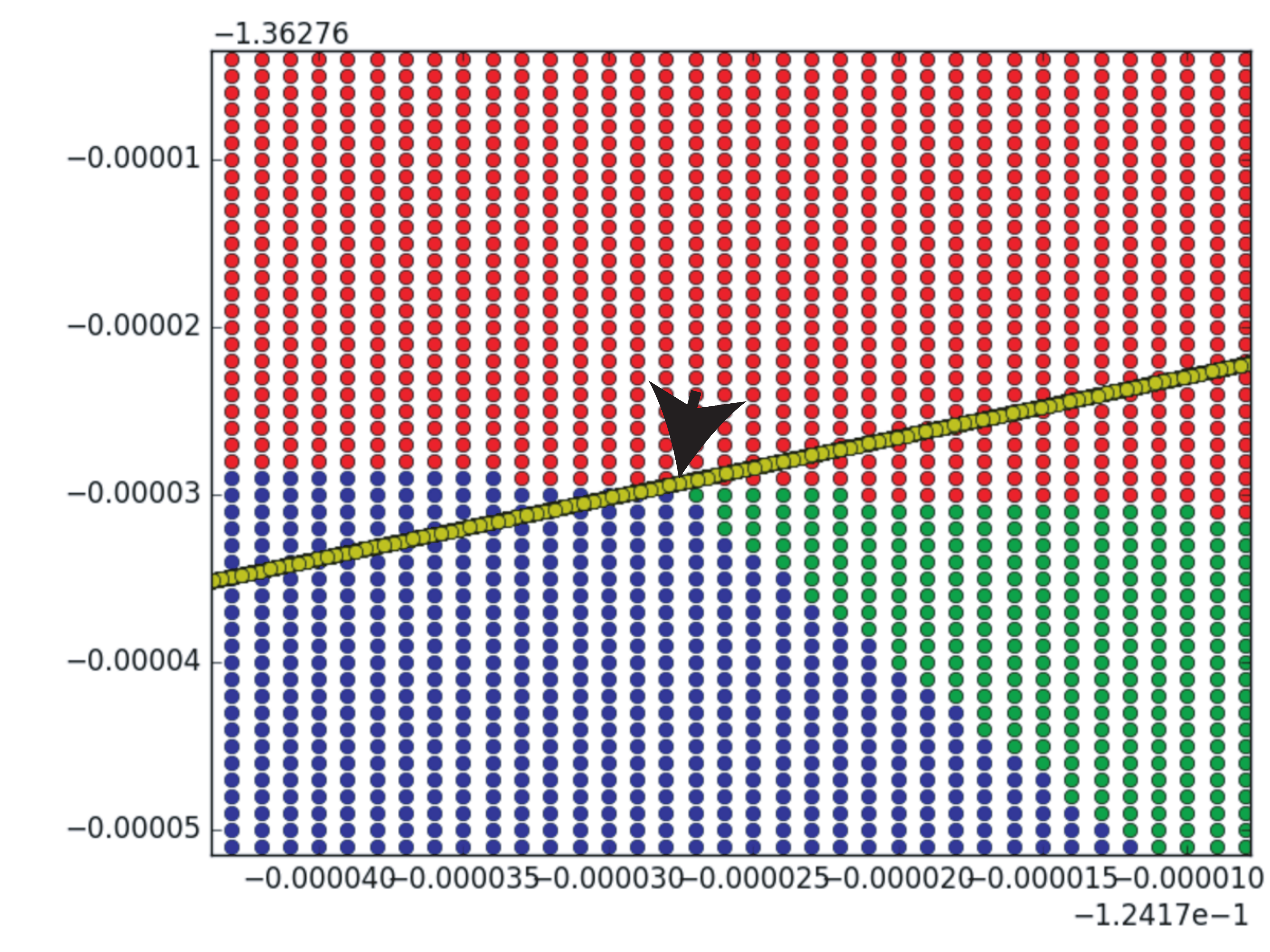}
\caption{Completely labeled square determined by the grid}
\label{fig:completely_labeled}
\end{figure}
\begin{figure}
\includegraphics[width=0.75\textwidth]{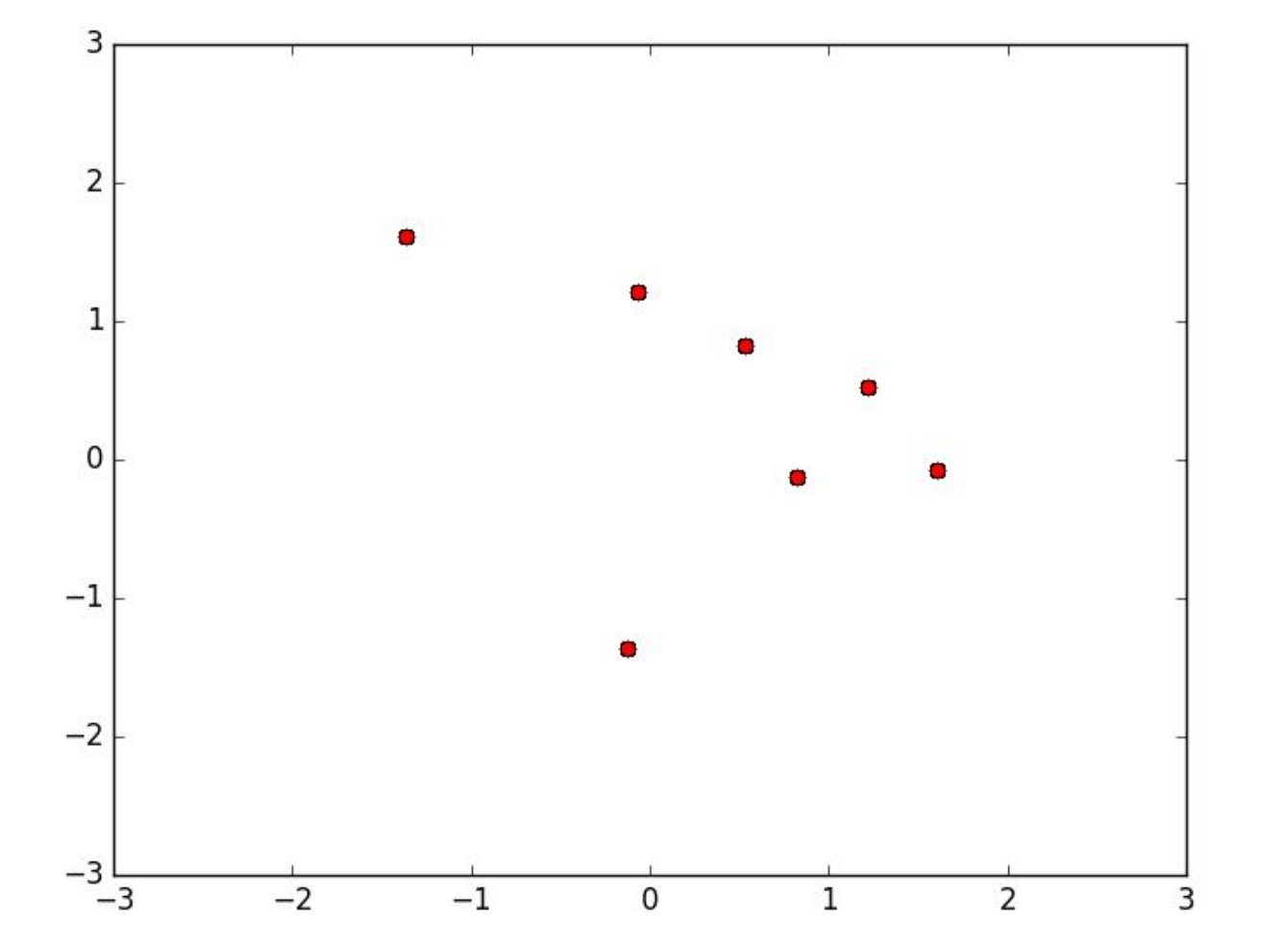}
\caption{Approximate period-$7$ orbit.}
\label{fig:period_seven}
\end{figure}

We build a window $D$ around the point $(-0.12,-1.36)$, which is a `first guess' of a period seven point, and compute its seventh iterate $f^7(D)$. See Fig. \ref{fig:window_seven_iterate}.
We define a fine grid on $D$ and we label the points of the grid according to \emph{Condition O}. We further reduce the labeling to only three labels $(1,1)\rightarrow 1$, $(-1,1)\rightarrow 2$, and $(1,-1), (-1,-1)\rightarrow 3$, as in Section \ref{sec:2D_windows}.  This  labeling is shown color coded in Fig.~\ref{fig:Sperner_labeling}. It is easy to see that the boundary of $D$ has a  { non-degenerate labeling} and $\ind_D(\phi)=1$, thus $D$ is correctly aligned to itself under $f^7$.

A completely labeled square in the grid decomposition occurs where the three different labels `meet'; this square has vertices at $x=-0.124198$, $y=-1.36279$ (red), $x=-0.124197$, $y=-1.36279$ (red),
$x=-0.124198$, $y=-1.36279$ (blue), $x=-0.124197$, $y= -1.36279$ (green); see Fig.~\ref{fig:completely_labeled}. The corresponding approximate period-$7$ orbit is shown in Fig.~\ref{fig:period_seven}. { It is easy to see that the above square has non-zero index. Thus, there exists a true period-$7$ orbit with an initial point near the square.}

\section*{Acknowledgement}
Both authors are  grateful to Meir Retter who helped with the computer code for the example in Section \ref{sec:application}.
The first author is grateful  to Zhonggang (Zeke)  Zeng, who made us aware of the  numerical analysis literature related to the Sperner Lemma, and to Kathleen Dexter-Mitchell, who read an early version of this work.

\bibliographystyle{alpha}

\bibliography{Sperner_biblio}
\end{document}